\documentclass[11pt,reqno]{amsproc}
\usepackage{graphicx}
\usepackage[margin=1in]{geometry}
\usepackage{amsmath, amsthm, amssymb}
\usepackage[colorlinks=true, pdfstartview=FitV, linkcolor=blue,citecolor=blue, urlcolor=blue,bookmarksdepth=2]{hyperref}
\usepackage[usenames,dvipsnames]{color}
\usepackage{mathtools}
\usepackage[backend=bibtex]{biblatex}
\usepackage[export]{adjustbox}
	
\usepackage{subcaption}
\usepackage{comment}
\excludecomment{comment}  %do not show comments
\usepackage{float}
\title[Periodic orbits in the Brusselator system]{Computer-assisted validation of the existence of periodic orbits in the Brusselator system}

\author[Jakub Bana\'{s}kiewicz, Piotr Kalita, Piotr Zgliczyński]{Jakub Bana\'{s}kiewicz$^{\dagger,\star}$, Piotr Kalita$^\star$, Piotr Zgliczyński$^\star$}
\address{$^\dagger$Faculty of Mathematics and Computer Science, Jagiellonian University, ul. Łojasiewicza 6, 30-348, Kraków, Poland}
\address{$^\star$AGH University of Science and Technology, al. Mickiewicza 30, 30-059 Kraków, Poland}
\email{banaskiewicz@agh.edu.pl, piotr.kalita@ii.uj.edu.pl, umzglicz@cyf-kr.edu.pl}
\thanks{The work of all three authors  was supported by National Science Center (NCN) of Poland under project No. UMO-2016/22/A/ST1/00077.  The research of JB for this publication has
	been supported by
	a grant from the Antropocen Priority Research Area under the Strategic
	Programme Excellence Initiative at
	Jagiellonian University. The work of PK was also partially supported by NCN of Poland under project No. DEC-2017/25/B/ST1/00302 and by Ministerio de Ciencia e Innovación of Kingdom of Spain under project No.
	 PID2021-122991NB-C21.}

\date{\today}

\newcommand{\sgn}{\text{sgn}}
\newcommand{\norm}[1]{\left\|{#1}\right\|}
\newcommand{\skalarProduct}[2]{\left<{#1},{#2} \right>}
\newcommand{\set}[1]{\{{#1}\}}

\newcommand{\R}[0]{\mathbb{R}}

\newtheorem{theorem}{Theorem}[section]

\newtheorem{lemma}[theorem]{Lemma}

\newtheorem{proposition}{Proposition}[section]
\theoremstyle{definition}
\newtheorem{definition}[proposition]{Definition}
\newtheorem{remark}[proposition]{Remark}

\numberwithin{equation}{section}

\begin{document}
\begin{abstract}
    We investigate the Brusselator system with diffusion  and Dirichlet boundary conditions on one dimensional space interval. Our proof demonstrates that, for certain parameter values, a periodic orbit exists. This proof is computer-assisted and rooted in the rigorous integration of partial differential equations. Additionally, we present the evidence of the occurrence of period-doubling bifurcation.
\end{abstract}
\maketitle
\input{}

\section{Introduction}
In this paper, we study the dynamics of the Brusselator system with diffusion, described by the following initial and boundary value problem consisting of two mutually coupled partial differential equations.
\begin{equation}\label{eq:BrusselatorPDE}
\begin{cases}
 u_t = d_1  u_{xx}- (B+1)u +u^2v + A \sin(x) \;  \text{for}\; (x,t)\in (0,\pi)\times(0,\infty),\\
 v_t = d_2 v_{xx} + Bu - u^2v  \;  \text{for}\; (x,t)\in (0,\pi)\times(0,\infty),\\
u(t,x)= v(t,x) = 0\; \text{for}\; (x,t)\in \{0,\pi\} \times (0,\infty),
\\u(0,x) = u^0(x),\ v(0,x) = v^0(x)\ \textrm{for}\ x\in(0,\pi).
\end{cases}
\end{equation}
The system models autocatalytic reactions in form
\begin{align*}
    &A \rightarrow u,
    \\&2u + v \rightarrow 3u,
     \\&B + u \rightarrow v + D,
     \\&u \rightarrow E.
\end{align*}
The coefficients $A,B$ in the system are given and they correspond to densities of the substances $A$ and $B$ in the above reactions. The unknown functions $u(t,x)$ and $v(t,x)$ describe the densities of two substances which we also denote by $u$ and $v.$ The system \eqref{eq:BrusselatorPDE} is an extension of the planar Brusselator ODE \eqref{eq:BrusselatorODE} where the substances $u$ and $v$ are homogeneously spread in the domain.
The terms $u_{xx}$ and $v_{xx}$ represent the diffusion of substances. The coefficients
$d_1,d_2$ are the corresponding diffusion rates.

If we drop the diffusion and the dependence on the variable $x$ in the term $A\sin(x)$ from the system \eqref{eq:BrusselatorPDE}, we obtain the following planar ODE
\begin{equation}\label{eq:BrusselatorODE}
	\begin{cases}
		u' = - (B+1)u +u^2v + A  \ \   \text{for}\ \  t\in \R,\\
		v' = Bu - u^2v  \ \   \text{for}\ \  t\in \R.
	\end{cases}
\end{equation}
The planarity of the above system implies that the invariant sets consist of fixed points, periodic orbits, and heteroclinic connections between them. In fact, it is known that in \eqref{eq:BrusselatorODE}, there can exist an attracting periodic orbit that arises via the Hopf bifurcation \cite[Theorem 3]{HopfBifurcationGeneralBrusselator}.
The analytical results about the Brusselator system of PDEs with diffusion are limited.
In the article \cite{YouBrusselatorAttractor} the existence of the global attractor for the Brusselator system on the 3-dimensional domain is proved. While this global attractor is known to exist, the question about its structure, which pertains to the understanding of the problem dynamics, remains unanswered. Some partial analytical results about the dynamics are available for Neumann boundary conditions, where the homogeneous steady state is known from the solutions of the corresponding ODE \eqref{eq:BrusselatorODE}. In this case, one can linearize the system of PDEs in its vicinity. Such results are available for example in \cite{BROWN19951713} and \cite{StabilityPeñaBegoña}.
The case with Dirichlet conditions, which we consider, appears to be much more challenging. In \cite{AUCHMUTY1975323}, the stability analysis of the steady state was carried out for Dirichlet non-homogeneous conditions. In such case there exists a nonzero homogeneous steady state. However, this type of analysis is not possible for the problem we are dealing with, as the homogeneous steady state does not exist (it would have to be equal to zero). Based on numerical observations, the system  \eqref{eq:BrusselatorPDE} possesses a periodic orbit for some range of parameters $d_1, d_2, A, B$. We anticipate that this periodic orbit arises from a mechanism similar to the one known for the planar ODE \eqref{eq:BrusselatorODE}, namely through a Hopf bifurcation.

We deal with the apparent impossibility of obtaining purely analytical results on the periodic orbit existence by using the computer assisted techniques.

Specifically, we perform a computer-assisted proof of the following theorem.
\begin{theorem} \label{th:BrusselatorPeriodicOrbit}
For the parameters $d_1 = 0.2$, $d_2 = 0.02$, $A = 1$, and $B= 2$, the Brusselator system has a time-periodic orbit.
%Moreover this periodic orbit is symmetric with respect to line $y = \frac{\pi}{2}$.
\end{theorem}
\begin{figure}[h]
    \includegraphics[width=0.9\linewidth, height=6cm]
    {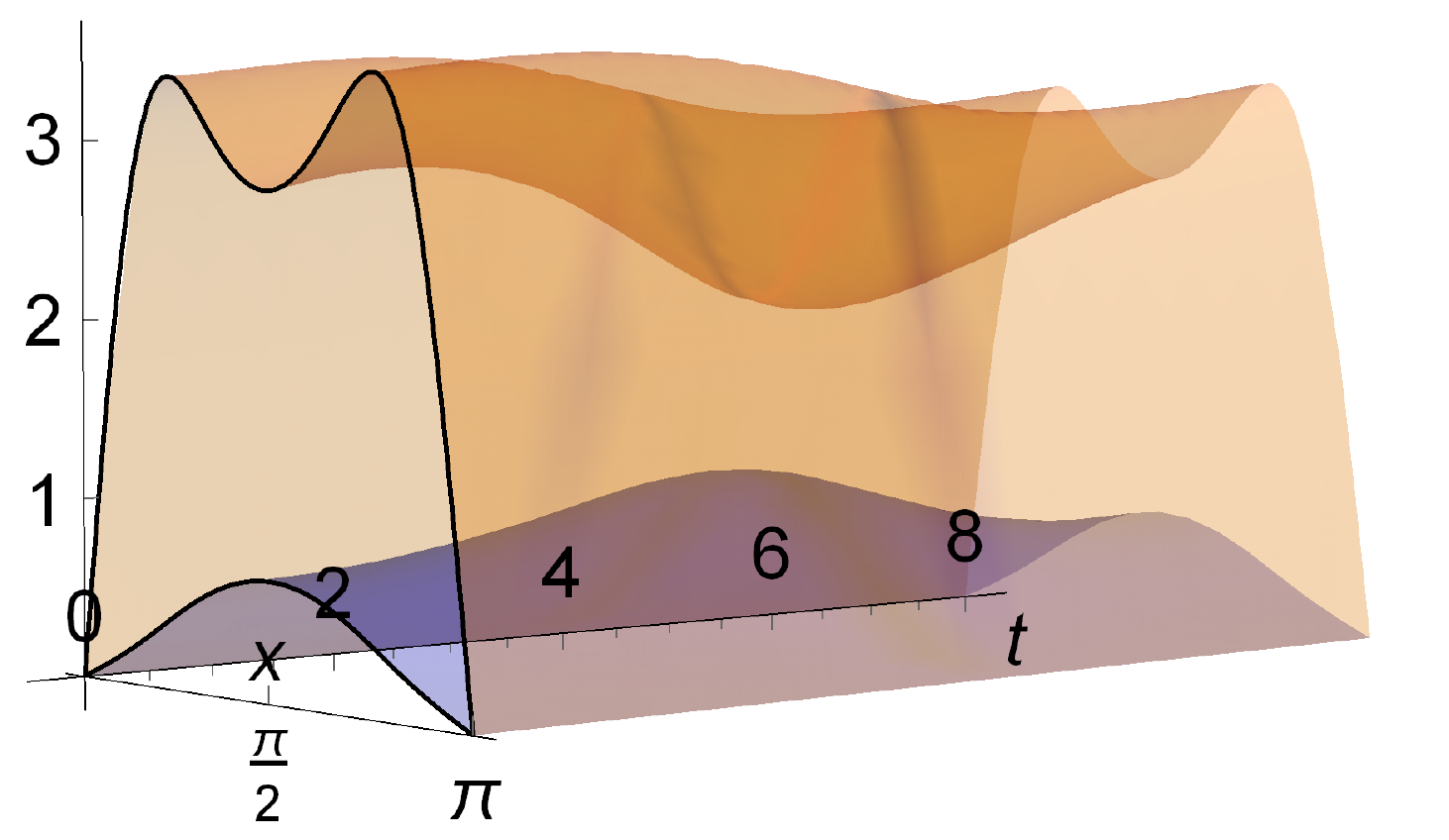}
\label{fig:image3}
\caption{The numerical approximation of the time-periodic orbit obtained in Theorem \ref{th:BrusselatorPeriodicOrbit}.
Blue and orange surfaces
correspond to plots of $u(t,x)$ and $v(t,x)$ respectively.
}
\end{figure}

It is worth noting that while numerical simulations suggest that this orbit from Theorem \ref{th:BrusselatorPeriodicOrbit} appears to be attracting, it has not been rigorously proven to be so.

The results on the existence of periodic orbit for the Brusselator were also obtained recently, together with the proof of the Hopf bifurcation, also using the computer assisted techniques, in the paper \cite{ArioliBrusselator}. The approach employed there is based on the Newton--Kantorovich method. The author demonstrates that a specific Newton-type operator has a fixed point, enabling him to establish the existence of a periodic orbit. This class of methods has been successfully applied for many problems governed by PDEs (see for example \cite{NSArioliHopf,NSLessBreden,EllipticSurvay,ElipticUnboundedPlum}).
The monograph \cite{ValidationBookPlum} contains detailed description and up to date overview of these methods together with numerous applications.
 In our  proof of Theorem \ref{th:BrusselatorPeriodicOrbit} we are using different method.  Namely, our approach   is based on the algorithm of rigorous forward integration of the dissipative systems.
This method was developed in articles  \cite{ZPKuramotoII,ZPKuramotoIII} and applied there for the Kuramoto--Shivasinski equation.
Similarly, as in the above works, we express the solution in terms of the Fourier series.
Specifically, as the solutions $u(t,x)$ and $v(t,x)$ satisfy the Dirichlet boundary conditions on $[0,\pi]$, they are represented as the sum of sine Fourier series
\begin{equation}\label{eq:decomposion}
u(t,x) = \sum_{i=1}^\infty u_i(t)\sin(i x),\quad
v(t,x) = \sum_{i=1}^\infty v_i(t)\sin(i x).
\end{equation}
The ability to work with all Fourier coefficients facilitates a straightforward and efficient integration of the Brusselator system. It is noteworthy, however, that in Theorem \ref{th:Arioli}, we obtain the same periodic orbit as in \cite{ArioliBrusselator}, cross-validating the accuracy of the two distinct methods.
Other approaches of representing the solution are also possible. In the work \cite{FEMPiotrek} the solution to the Burgers equation was represented by the first order finite element basis together with the estimates on the norms in the Sobolev spaces.

Note that the constant term $A$ from the ODE \eqref{eq:BrusselatorODE} has been replaced by the term $A\sin(x)$ in our PDE extension of the Brusselator system. In principle, we could consider the PDE version of the problem with the constant term $A$, that is the following system 
\begin{equation}\label{eq:BrusselatorPDEILL}
\begin{cases}
 u_t = d_1  u_{xx}- (B+1)u +u^2v + A  \;  \text{for}\; (x,t)\in (0,\pi)\times(0,\infty),\\
 v_t = d_2 v_{xx} + Bu - u^2v  \;  \text{for}\; (x,t)\in (0,\pi)\times(0,\infty),\\
u(t,x)= v(t,x) = 0\; \text{for}\; (x,t)\in \{0,\pi\} \times (0,\infty),
\\u(0,x) = u^0(x),\ v(0,x) = v^0(x)\ \textrm{for}\ x\in(0,\pi).
\end{cases}
\end{equation}
In such a case, we observe that $u_{xx}(t,x) = -\frac{A}{d_1}\not = 0$ for the boundary points  $x\in \{0,\pi\}.$ and for every $t\in\mathbb{R}^+,$ which means that compatibly conditions are not met. 
Let us represent $u$ in term of the sine Fourier series \eqref{eq:decomposion}. Assume that the series $\sum_{i=1}^\infty |u_i(t)|i^2$ is convergent.  Then, we can differentiate the Fourier series  twice, and we obtain 
\begin{equation*}
    u_{xx}(t,x) = \sum_{i=1}^\infty -u_i(t)i^2\sin(i x)
\end{equation*}
Therefore $u_{xx}(t,x) = 0$ for $x\in\{0,\pi\},$ which is not true, and hence, by contradiction, the series  $\sum_{i=1}^\infty |u_i(t)|i^2$ has to diverge.   Therefore, the Fourier series cannot converge to the solution fast. This unwelcome effect does not occur when we consider term $A\sin(x)$ instead of $A$ in the system \eqref{eq:BrusselatorPDE}, because the function $Asin(x)$ on $[0,\pi]$ is a restriction of a smooth, odd and $2\pi$ periodic function.

The novelty of the present paper is the proof of the periodic orbit existence for the Brusselator systems using the rigorous forward integration techniques. This integration algorithm is, according to our knowledge, applied by us for the first time for a system of PDEs: in our case two mutually coupled nonlinear parabolic PDEs with polynomials of order $3$ in the nonlinear term.  We  underline that the rigorous integration scheme which we use is the same as in  \cite{ZPKuramotoII, ZPKuramotoIII}, where it was used for the Kuramoto--Sivashinski equation with odd-periodic boundary conditions. We show its applicability for the problem with higher degree of nonlinearity: is our case the two equations of the system are coupled through the cubic term, while the nonlinearity in the Kuramoto--Sivashinsky equation is quadratic. The key concept which makes it possible for the integration scheme to work is the same in our case as in  \cite{ZPKuramotoII, ZPKuramotoIII}. Namely, the dissipativity of the leading linear operator together with appropriate a priori estimates for the nonlinearity, which is of lower order, allow the linear  terms to dominate over the nonlinear ones at appropriately high modes in  the Fourier expansion. This allows us to treat the tail of the Fourier expansion uniformly, by controlling a polynomial decay of the coefficients in every time step, cf. Lemma~\ref{lem:whyItWorks}.  Our techniques hold potential for wider applicability. To this end, in Section \ref{sec:algebra}, we establish estimates on the convolution of sine and cosine Fourier series, which can be utilized to calculate nonlinear terms for a general dissipative system with polynomial nonlinearities in one spatial dimension. These estimates are a crucial component in rigorous integration algorithms for such systems.
    A further innovation is that in our theoretical results, which ensure the validity of the algorithm, we don't require the use of Galerkin projections of the solution. Instead, we work directly with the solution of the PDE system on the level of abstract theorems as can be seen in Sections \ref{sec:AlgorithmIntegration} and \ref{sec:comAssisedProof}. This simplifies our assumptions and hence it makes the results more accessible to the dissipative PDE community. Additionally, we examine the limitations of our algorithm for PDEs with nonlinearities that do not meet the compatibility condition on the boundary:   this is the case even with the constant $A$ in place of  $A\sin(x)$ in the equation for $u_t$ in the Brusselator system \eqref{eq:BrusselatorPDE}. We illustrate this issue by considering the problem governed by the diffusive logistic equation \eqref{eq:logisticEquation} as an example.

 In Section \ref{sec:numerical}, we observe that for a sufficiently large parameter $B$, the system exhibits slow-fast behavior, as expected for the Brusselator system and known in the ODE case. This effect is stronger for higher Fourier modes. To demonstrate this, we establish the existence of periodic orbits for parameters $d_1 = 0.2,\;d_1 = 0.02,\;A=1,\;B=2+\frac{i}{10}$ for $i\in{0,\dots,11}$. Figures \ref{fig:OrbitsPlots1} and \ref{fig:OrbitsPlots2} show that some of these orbits exhibit slow-fast behavior.

 In \cite{ArioliBrusselator} Arioli observed a period doubling bifurcation, a phenomenon that cannot occur in the planar ODE \eqref{eq:BrusselatorODE}. Thus, the dynamics of \eqref{eq:BrusselatorPDE} is expected to be more complicated than that of the planar ODE \eqref{eq:BrusselatorODE}. Although we do not rigorously prove the bifurcation, we also show that the minimal period of the found orbits approximately doubles with a small increase in the parameter $B$, as seen in Theorem \ref{th:podwojenieOkresu}. In this range of parameter $B$ there should  also exist one unstable orbit with the period that is not doubled. To prove the existence of this kind of orbit we could use the rigorous solver for the Brusselator system together with the concept of $h$-sets and covering relations (see for example \cite[Section 2.1]{KuramotoAutomaticDiff} and \cite[Section 10.2]{ZPKuramotoIII}). Such approach has been successfully used to prove the existence of unstable periodic orbits before \cite[Theorem 45, Theorem 46]{ZPKuramotoIII}).

Other nontrivial dynamics of the Brusselator system, such as the existence of 2-dimensional attracting tori and chaos, were numerically investigated in \cite{BrusselatorChaosAnd2DTori}. Our numerical observations support the existence of 2-dimensional attracting tori for small diffusion parameters, although a rigorous proof of their existence remains an open problem. There are many avenues for further research on this topic. Numerical simulations indicate that the periodic orbit established in Theorem \ref{th:BrusselatorPeriodicOrbit} is attracting. However, providing a rigorous computer-assisted proof of this observation is challenging, as it requires a rigorous $C^1$ calculation, i.e. the integration of the variational equation for the Brusselator system.

The structure of the article is as follows.
In  Section \ref{sec:AlgorithmIntegration} we describe the algorithm of rigorous integration for dissipative equations. In Section \ref{sec:comAssisedProof} we describe the computer assisted proof of  Theorem
\ref{th:BrusselatorPeriodicOrbit}.
 In Section \ref{sec:crossing}, we address the algorithm for computing the Poincaré map and prove Theorem \ref{th:PoincareFixedPoint}, which pertains to the fixed point of this map which corresponds to the periodic orbit of the system. In the remaining part of Section \ref{sec:comAssisedProof} we describe the validation of the assumptions of this theorem for the Brusselator system.
 Section \ref{sec:numerical} contains numerical and rigorous results for various parameters of the Brusselator system.
Finally, in Section \ref{sec:algebra}, we provide results on the algebra of infinite series, which is utilized in the algorithms.

The code which was used in computer assisted proofs is published at GitHub \cite{Code} and based on  CAPD library \cite{CAPD,CAPDArticle}.

\begin{comment}

Brusselator system in ODE case has a periodic orbit for some parameters which arises from Hopf biffurcation of stationary point.
In this paper we are proving that periodic orbit also exists for the system with diffusion and Dirichlet boundary conditions. Prove is computer assisted.
\end{comment}

%\input{Wersja 5/problemDef}
\section{Algorithm of integration}\label{sec:AlgorithmIntegration}
In this section we  present our version of the technique of integration for infinite dimensional dissipative systems proposed in \cite{ZPKuramotoII}, where it has been used for the Kuramoto--Shivashinsky equation. This method relies of the rigorous integration of a differential inclusion and it can be used for many dissipative problems in mathematical physics. We  discuss it in the abstract setting but some details will be specified for the Brusselator system. The other approach based on the automatic differentation is presented in \cite{KuramotoAutomaticDiff}. We summarize the content of this section. We start, in Section \ref{sec21}, with the formulation of the abstract problem for which the algorithm can be applied, and in Section \ref{sec:abbru} we discuss its realization for Brusselator PDEs.
In Section \ref{sect22} we briefly describe the goal of the algorithm and its main steps. The sets of states and their representation are discussed in Section \ref{sec:descritionSet}, and the way to compute the nonlinearities present in the system on those sets is described in Section \ref{sect24}.
In Section \ref{sec:enclosure}, we outline the steps for determining the enclosure and provide a justification for its correctness. The algorithm of evolution of sets is discussed in Section \ref{sec:Evolving set}.

\subsection{Abstract Problem}\label{sec21}

Let $H$ be a real Hilbert space with the scalar product $\skalarProduct{.}{.}$
and $Y$ a be a Banach space which continuously and densely embeds in $H$, that is $\norm{x}_H\leq C\norm{x}_Y$ for $x\in Y.$
We assume that  $\{e_1,e_2,\ldots\}$ is an orthogonal basis of $H$ such that $e_i\in Y$ for every $i\in\mathbb{N}.$

For a given $x\in H$ by $x_i$ we will denote the Fourier coefficient $x_i = \frac{\skalarProduct{e_i}{x}}{\skalarProduct{e_i}{e_i}}.$
We consider following problem
\begin{equation}\label{eq:AbstractProblem}
\begin{cases}
   \frac{d}{dt}x(t) = Lx(t) + f(x(t)) = F(x(t)),
   \\ x(0) = x^0,
\end{cases}
\end{equation}
where $L$ is a diagonal operator such that $Le_i=\lambda_i e_i$ where $\lambda\in\mathbb{R}\setminus\{0\}$, which generates a $C^0$ semigroup $e^{tL}:Y\to Y$.
We have $e^{Lt}e_i = e^{\lambda_it}e_i .$
We assume that $x^0\in Y$, and $f:Y\to Y$ is a continuous mapping. For a given $x\in Y$, we will use the notation $f_i(x)$ and $F_i(x)$ for $ \frac{\langle f(x),e_i\rangle}{\skalarProduct{e_i}{e_i}}$ and   $f_i(x) + \lambda_i x_i$, respectively.

The following lemma provides the criteria for local in time existence and uniqueness of mild solutions to problem \eqref{eq:AbstractProblem}. \begin{lemma}\label{lem:solExistence}
Assume that
\begin{enumerate}
    \item[(A1)] For some $C_1>0$ there holds $ \norm{e^{tL}}_Y \leq e^{C_1 t}.$
    \item[(A2)] For every $R>0$ there exists $C(R)>0$ such that for every $x,y\in Y$ such that $\norm{x}_Y,\norm{y}_Y\leq R$  there holds $\norm{f(x) - f(y)}_{Y}\leq C(R) \norm{x-y}_{Y}.$
\end{enumerate}
Then for every initial data $x^0\in Y$  there exists the unique local in time solution to problem \eqref{eq:AbstractProblem} understood in the following sense
\begin{equation}\label{eq:Duhamel}
x(t) = e^{Lt}x^0 + \int_0^t e^{L(t-s)}f(x(s))\, ds,
\end{equation}
where the equality in $Y$ is supposed to hold for every $t\in [0,T]$, where $T$ may depend on $x_0$.

\end{lemma}
Instead of proving Lemma  \ref{lem:solExistence}, we will show more general result which implies it. Namely, we can replace (A1) and (A2) with  more general conditions. The following result generalises Lemma \ref{lem:solExistence} to the case when is $f$ is a continuous map from $Y$ to $Y^1,$ where $Y^1$ is a Banach space such that $Y$ is continuously embedded in $Y^1,$ and $Y^1$ is continuously embedded in $H.$ This more general result is useful, for example, if the nonlinear term in the problem depends not only on the value of the solution but also on the values of its spatial derivatives, which is the case for the Burgers or Kuramoto--Shivashinsky equations.
\begin{lemma}
Assume that
\begin{enumerate}
    \item[(B1)] For some $C_1,C_2>0$ and every $t\in[0,\infty)$ there holds $ \norm{e^{tL}}_Y \leq  C_1 e^{C_2 t}.$
    \item[(B2)] For every $R>0$ there exists $C(R)>0$ such that for every $x,y\in Y$ such that $\norm{x}_Y,\norm{y}_Y\leq R$  there holds $\norm{f(x) - f(y)}_{Y^1}\leq C(R) \norm{x-y}_{Y}.$
    \item[(B3)]  The semigroup $e^{tL}$ can be extended to the $C^0$ semigroup on $Y^1$. There exist constants $C_3,C_4>0$ and $\gamma\in[0,1)$ such that for every $z\in Y^1$ and $t\in(0,\infty)$ there holds $\norm{e^{tL}z }_Y\leq C_3e^{C_4 t}\frac{1}{t^\gamma}\norm{z}_{Y^1}$.
\end{enumerate}
Then, for every initial data $x^0 \in Y$, there exists a unique time-local solution to problem \eqref{eq:AbstractProblem} in the following sense:
$$
x(t) = e^{Lt}x^0 + \int_0^t e^{L(t-s)}f(x(s))\, ds,
$$
where the equality holds in $Y$ for all $t \in [0, T]$, where $T$ may depend on $x_0$.
\end{lemma}
\begin{proof}
For a given $x^0\in Y$ and $\delta>0$ consider the set
\begin{equation*}
    S_\delta:= \{y\in C([0,\delta];Y)\colon y(0) = x^0\;  \text{and for every} \ \ t\in[0,\delta] \text{ we have } \norm{y(t)}_Y\leq 1+C_1\norm{x^0}_Y  \},
\end{equation*}
and define the mapping  $T:C([0,\delta];Y) \to C([0,\delta];Y)$ by the formula
\begin{equation*}
    T(y)(t) = e^{Lt}x^0 + \int_0^te^{L(t-s)}f(y(s))ds.
\end{equation*}
The space $C([0,\delta];Y)$ is equipped with the norm $\sup_{t\in[0,\delta]}\norm{y(t)}_{Y}$.
We have
\begin{align*}
    \norm{T(y)(t)}_Y
    &\leq
    C_1e^{C_2t}\norm{x^0}_{Y}+
    \int_0^t\norm{e^{L(t-s)} f(y(s))}_Y ds
    \\&\leq
    C_1e^{C_2t}\norm{x^0}_{Y}+
    C_4\int_0^t e^{(t-s)C_3}\frac{1}{(t-s)^{\gamma}}\norm{ f(y(s))}_{Y_1} ds
    \\&\leq
     C_1e^{C_2t}\norm{x^0}_{Y}+
    C_4 e^{tC_3}\int_0^t\frac{1}{(t-s)^{\gamma}} ds
     \sup_{s\in[0,t]}\norm{f(y(s))}_{Y_1}
     \\&\leq
     C_1e^{C_2t}\norm{x^0}_{Y}+
     \frac{t^{1-\gamma}}{1-\gamma}
    C_4  e^{tC_3}
    \left(C(R)\sup_{s\in[0,t]}\norm{y(t)}_Y+\norm{f(0)}_{Y_1} \right),
\end{align*}
where $R = 1+C_1\norm{x^0}_Y.$
If we pick $\delta>0$ such that
\begin{equation*}
  e^{C_2\delta}\leq 1+\frac{1}{2C_1\norm{x_0}_Y},
  \quad\text{and}\quad
    \delta^{1-\gamma}e^{\delta C_3}\leq \frac{1-\gamma}{2C_4\left(C(R)R+\norm{f(0)}_{Y_1} \right)}
\end{equation*}

we have that $T(S_\delta) \subset S_\delta.$ We have also
\begin{align*}
  \norm{T(y_1)(t)-T(y_2)(t)}_Y \leq  C_4e^{tC_3}C(R)\frac{t^{1-\gamma}}{1-\gamma}
  \sup_{s\in[0,t]} \norm{y_1(s)-y_2(s)}_Y.
\end{align*}
If we take $\delta$ such that
\begin{equation*}
    \delta^{1-\gamma}e^{\delta C_3}\leq \frac{1-\gamma}{2C_4C(R)},
\end{equation*}
the mapping $T$ is also a contraction in the set $S_\delta.$ From Banach fixed point theorem we that $T$ has a unique fixed point which is a solution to \eqref{eq:AbstractProblem}.
\end{proof}
\begin{lemma}\label{re:flowCon}
Assume $(A1)-(A2)$ or $(B1)-(B3)$. For every $x^0\in Y$ there exist $t_{max}(x^0)\in(0,\infty]$ such that
for the
unique solution $x:[0,t_{max}(x^0)) \to Y$ of \eqref{eq:AbstractProblem} the interval $[0,t_{max}(x^0))$ is maximal interval of existence of this solution. Additionally if we consider the set
$$\Omega :=\left\{(t,x)\in \R^+\times Y: t\in\left[0,t_{max}(x)\right)\right\},$$ then the function $\varphi:\Omega\to Y$ given by  formula
\begin{equation}
    \varphi(t,x^0) = e^{Lt}x^0 + \int_0^t e^{L(t-s)} f(x(s))\, ds.
\end{equation}
defines a local semigroup.
\end{lemma}
\begin{remark}\label{lem:AbstractFourierCoef}
Let $x:[0,\tau]\to Y$ satisfy \eqref{eq:Duhamel}.
Assume $(A1)-(A2)$ or $(B1)-(B3)$.
Then for every $i\in\mathbb{N},$ the Fourier coefficients $x_i$ satisfy the following non-autonomous ODE
\begin{equation*}
    \frac{dx_i}{dt}(t) = \lambda_ix_i(t)+f_i(x(t)),
    \quad \text{for every $t\in(0,\tau)$}.
\end{equation*}
\end{remark}
\begin{proof}
If $x(t)$ satisfies \eqref{eq:Duhamel} then for every  $i\in\mathbb{N}$ we have
\begin{equation}\label{eq:formula1}
    x_i(t) = e^{t \lambda_i}x^0_i+
    \int_0^te^{(t-s)\lambda_i} f_i(x(s))ds.
\end{equation}
Observe that if $(A1)-(A2)$ or $(B1)-(B3)$ hold, then for every $i\in\mathbb{N}$ the function $f_i:Y\to\mathbb{R}$ is locally Lipschitz.
Hence, in the formula \eqref{eq:formula1} the function under integral is continuous.
Therefore we can differentiate this formula and get the ODE for $i$-th Fourier coefficient.
\end{proof}

\subsection{The Brusselator system.} \label{sec:abbru} We discuss how to represent the Brusselator system in the abstract framework presented in the previous section. We use the notation  $L^2$ for $L^2(0,\pi)$ equipped with the norm $\norm{u}_{L^2} = \sqrt{\int_0^{\pi}u(x)^2\, dx}$ and $C_0$ for $\{ u\in C([0,\pi])\,:\ u(0)=u(\pi) = 0\}$ equipped with the norm $\|u\|_{C_0} = \max_{x\in [0,\pi]}\{|u(x)|\}$, and we consider the following two product spaces: the Hilbert space $$
H = L^2\times L^2 \ \ \textrm{with the norm} \ \ \norm{(u,v)}_H^2 = \norm{u}_{L^2}^2 + \norm{v}_{L^2}^2 \ \ \textrm{for}\ \  (u,v)\in H$$ and the Banach space
$$Y = C_0\times C_0\ \ \textrm{with the norm}\ \ \norm{(u,v)}_{Y}
= \max\{\norm{u}_{C_0},\norm{v}_{C_0}\}
\ \ \textrm{for}\ \  (u,v)\in Y.$$
In the space $H$ system of functions  $\{e_k\}_{k=1}^\infty$ defined in following way
\begin{equation*}
    e_{2k-1} = (\sin(kx),0)\quad e_{2k} =(0,\sin(kx))\ \ \textrm{for}\ \  k\in \{1,2,\ldots\}.
\end{equation*}
is the orthogonal basis.  For $u,v\in L^2$ we denote
by $u_k$ and $v_k$ the $k$-th coefficients in the Fourier expansion in the sine series, of $u$ and $v$ respectively.
We define the operator
$$
L:D(L)\supset Y \to Y\ \ \textrm{as}\ \ L(u,v) = (d_1u_{xx} - (B+1)u, d_2 v_{xx}),$$
where $D(L) = \{ (u,v)\in H^1_0 \times H^1_0 \colon \  L(u,v) \in Y\}$. The operator $L$ defines a $C^0$ semigroup on $Y$, denoted by $e^{tL}$, cf. \cite[Proposition 2.6.7 and Theorem 3.1.1]{HarauxCazenave}. Observe that
$Le_{2k-1} = -(d_1k^2 +B+ 1) e_{2k-1}$ and $Le_{2k} = -d_2k^2e_{2k} ,$
so  $\{e_k\}_{k=1}^\infty$ are the eigenfunctions of $L.$  We define $f(u,v) = (u^2v + A\sin(x),Bu-u^2v).$
We can write the Brusselator system \eqref{eq:BrusselatorPDE} as the following abstract problem
\begin{equation}\label{eq:AbstractBrusselator}
    \begin{cases}
     \frac{d}{dt}(u(t),v(t)) = L(u(t),v(t)) + f(u(t),v(t)),\\
      (u(0),v(0)) = (u^0,v^0).
    \end{cases}
\end{equation}

We apply Lemma \ref{lem:solExistence} to the above system, which gives the following result.
\begin{theorem}\label{ref:exbru}
For every $(u^0,v^0)\in Y$ there exists a function $(u,v):[0,t_{max}(u_0,v_0))\to Y$ which is the unique solution to the
\eqref{eq:AbstractBrusselator}, satisfying the Duhamel formula
\begin{equation*}
    (u(t),v(t)) = (u^0,u^0) + \int_0^te^{L(t-s)}f(u(s),v(s))ds.
\end{equation*}
\end{theorem}
\begin{proof}
We demonstrate that for the Brusselator problem \eqref{eq:AbstractBrusselator} assumptions (A1) and (A2) hold and we can use Lemma \ref{lem:solExistence}. Namely
\begin{equation*}
   \norm{e^{Lt}}_Y\leq 1,\text{ for every }t\in[0,\infty),
\end{equation*}
which follows from the maximum principle for the heat equation. Furthermore
\begin{equation*}
    \norm{u^2v-\Bar{u}^2\Bar{v}}_{C_0}
    \leq
    \norm{u}^2_{C_0}\norm{v-\Bar{v}}_{C_0}+
    \norm{u+\Bar{u}}_{C_0}
    \norm{\Bar{v}}_{C_0}
    \norm{u-\Bar{u}}_{C_0}\ \ \textrm{for}\ \ u,v\in C_0.
\end{equation*}
So for every $R>0$ there exist $C(R)$ such that for every $\norm{(u,v)}_Y,
\norm{(\Bar{u},\Bar{v})}_Y\leq R$
\begin{equation*}
  \norm{f(u,v)-f(\Bar{u},\Bar{v})}_Y\leq C(R)
  \norm{(u,v)-(\Bar{u},\Bar{v})}_Y,
\end{equation*}
which concludes the proof.
\end{proof}
From Remark \ref{lem:AbstractFourierCoef} and the formula for expanding the expression $u^2v$ in terms of sine Fourier series, we obtain the following lemma.

\begin{lemma}
Let  $(u(t),v(t))$ solve the system \eqref{eq:BrusselatorPDE}. For every $k\in\mathbb{N}$ there holds
\begin{equation}\label{eq:BrusselatorInfODE}
  \begin{cases}
   \frac{d}{dt}u_k = -u_k (d_1k^2 + 1 + B) + N(u,v)_k + A\delta_{1k},
   \\\frac{d}{dt}v_k =-v_k d_2k^2  + u_kB - N(u,v)_k,
  \end{cases}
  \end{equation}
  where
 \begin{equation}
       N(u,v)_k = \frac{1}{4}
\sum_{i_1 + i_2 + i_3 = k}u_{|i_1|}u_{|i_2|}v_{|i_3|} \sgn(-i_1i_2i_3) \text{ with $i_1,i_2,i_3\in\mathbb{Z}\setminus\{0\}$}.
\end{equation}
\end{lemma}
Finally, we present the result that states the fact that the space of functions with only odd Fourier coefficients being nonzero for functions $u$ and $v$ is forward-invariant and corresponds to the space of functions that are symmetric with respect to the point $x=\frac{\pi}{2}$.
\begin{proposition}\label{prop:OddSubspace}
Space
$W := \{(u,v)\in Y: u_{2i} = v_{2i} = 0, \;\text{for}\;i\in\mathbb{N}\}$ is invariant for system \eqref{eq:BrusselatorPDE}. Specifically, if $(u_0,v_0)\in W$ then $u(t,\frac{\pi}{2}+x) = u(t,\frac{\pi}{2}-x)$ and
$v(t,\frac{\pi}{2}+x) = v(t,\frac{\pi}{2}-x)$
for every $t\in [0,t_{max}(u_0,v_0))$ and almost every $x\in[0,\frac{\pi}{2}].$
\end{proposition}

\begin{remark}
In the proof of Theorem \ref{ref:exbru} we verify that assumptions (A1) and (A2) hold for the Brusselator system. If the lower order nonlinearity depends on the derivatives of the unknown, then we need  (B1)--(B3). Indeed, consider the Kuramoto--Sivashinsky equation 
$$
u_t = - \nu u_{xxxx} -  u_{xx} + (u^2)_x\ \ \textrm{for}\ \ (x,t)\in (0,\pi)\times (0,\infty), 
$$ 
with odd-periodic boundary conditions $u(0,t) = u(\pi,t)=u_{xx}(0,t) = u_{xx}(\pi,t)=0$ studied in \cite{ZPKuramotoII, ZPKuramotoIII}. We assume that the constant $\nu$ is positive. The system $\{ \sin(kx)\}_{k=1}^\infty$ constitutes the orthogonal in $L^2(0,\pi)$ basis of eigenfunctions of the leading linear operator $Lu = -\nu u_{xxxx}- u_{xx}$ with the eigenvalues $\lambda_k = -\nu k^4 + k^2$. To verify (B1)-(B3) we take $H=L^2(0,\pi)$, $Y^1=\{u\in H^3(0,\pi)\,: u(0)=u(\pi)=u_{xx}(0)=u_{xx}(\pi) = 0\}$, and $Y=H^4(0,\pi)\cap Y^1$. If $u=\sum_{k=1}^\infty u_k\sin(kx)$, then 
$$
\|u\|_{Y^1}^2 = \frac{2}{\pi}\sum_{k=1}^\infty k^6|u_k|^2,\ \  \|u\|_{Y}^2 = \frac{2}{\pi}\sum_{k=1}^\infty k^8|u_k|^2.
$$  
The operator $L$ is diagonal and the evolution of the $k$-th mode via the linear semigroup $e^{tL}$ is given by the formula
$$
u_k(t) = u_k(0)e^{(-\nu k^4+k^2)t}.
$$
It is easy to verify that the function $k\to -\nu k^4 + k^2$ has its maximum equal to $\frac{1}{4\nu}$ at $k=\frac{1}{\sqrt{2\nu}}$. This leads to the estimates $\|e^{tL}\|_{\mathcal{L}(Y^1;Y^1)}\leq e^{\frac{t}{4\nu}}$ and $\|e^{tL}\|_{\mathcal{L}(Y;Y)}\leq e^{\frac{t}{4\nu}}$. Verification of (B3) follows the concept of \cite[Lemma 3.1]{Tadmor}. Indeed, assuming that $u^0\in Y^1$ is the initial data, we obtain
$$
\|e^{tL}u^0\|^2_{Y} = \frac{2}{\pi}\sum_{k=1}^\infty k^6|u^0_k|^2 k^2e^{2(k^2-\nu k^4)t}.
$$
A straightforward computation which involves the maximization over $k\geq 0$ shows that 
$$
ke^{(k^2-\nu k^4)t} \leq \frac{C}{\sqrt[4]{\nu t}}e^{\frac{t}{\nu}}, 
$$
where $C$ is independent of $k, \nu, t$. We deduce that
$$
\|e^{tL}u^0\|_{Y} \leq \frac{C}{\sqrt[4]{\nu t}}e^{\frac{t}{\nu}} \|u^0\|_{Y^1},
$$
and (B3) is proved. To get (B2) it is enough that
$$
\|2uu_x-2vv_x\|_{H^3} \leq C(R)\|u-v\|_{H^4},
$$
where $\|u\|_{H^4}, \|v\|_{H^4}\leq R$, which is straightforward to verify. 
\end{remark}

%\end{proof}
\subsection{Overview of algorithm}\label{sect22}
 First of all, since the phase space $Y$ of our  abstract problem \eqref{eq:AbstractProblem} is infinite dimensional, we need a suitable representation of sets from this space. Once we have such representation, the key concept of the method is, for a given set $X(0)$ of initial data and a time-step $\tau$, to effectively construct another set $X(\tau)$, such that it is guaranteed that every solution starting from $X(0)$ at time $t=0$ belongs to $X(\tau)$ at time $t=\tau$. Of course sets $X(0)$ and $X(\tau)$ must be described in previously defined representation. In other words, if $\varphi:\Omega\to \mathbb{R}$ is the local semigroup governed by the solutions of \eqref{eq:AbstractProblem} we need to be able to construct the set $X(\tau)$ such that $\varphi(\tau ,X(0)) \subset X(\tau)$. Furthermore, as we do not exclude the possibility of blow-up arbitrarily, the algorithm should also ensure that the value of $\varphi(\tau,x)$ is well-defined for every $x \in X(0)$.
 At the same time, the set $X(\tau)$ should be as small as possible, as we iterate the above procedure to find sets that contain all solutions originating from a given set of initial data at large times.
The chosen representation always results in an overestimation at each iteration. Because we need to represent sets $X(0)$ and $X(\tau)$ in the computer memory, which is finite, we represent those sets as finite objects.
In the abstract problem \eqref{eq:AbstractProblem} we assume that  our phase space $Y$ is embedded in a Hilbert space $H$ with the basis $\{e_k\}_{k=1}^\infty .$ We represent $H = H_P\oplus H_Q.$
where $H_P = \text{span}\{e_1,\dots,e_n\} \cong \mathbb{R}^n$ and $H_Q$ is an orthogonal complement of $H_P$ in $H.$ By $P,Q$ we will denote orthogonal projections on the spaces $H_P$ and $H_Q$ respectively.
We represent the sets $X(0), X(\tau) \subset Y$ as
\begin{equation*}
    X(0) =  X_{P}(0)  + X_{Q}(0)\qquad X(\tau) =  X_{P}(\tau) + X_{Q}(\tau),
\end{equation*}
where $X_{P}(0), X_{P}(\tau) \subset H_P$ are sets in finite dimensional space and $X_{Q}(0),X_{Q}(\tau) \subset H_Q$ are infinite dimensional sets, which need some finite representation. We realize this representation by giving some inequalities, which are uniform with respect to coefficients in the Fourier representation. Such decomposition of the set will be called a $P, Q$ representation.
Now, essentially, we divide algorithm into two parts
\begin{enumerate}
    \item  Find the set $X([0,\tau])$, given in the same $P, Q$ representation as the set  $X(0)$, such that every the solution to the considered problem satisfies $\varphi(t ,X(0)) \subset X([0,\tau])$ for every $t\in [0,\tau]$.  We will equivalently write $\varphi([0,t] ,X(0)) \subset X([0,\tau])$ and we will call such set an enclosure.
    The details of finding this enclosure are given in  Section \ref{sec:enclosure}.

    \item Use the obtained enclosure to find the set  $X(\tau).$ To this end, we use the following procedures to find separately the sets $X_P(\tau)$ and $X_Q(\tau).$

    \begin{itemize}
        \item
        We formulate a differential inclusion in $\mathbb{R}^n$ for the $P$ component of the solution. In practice this $P$ component consists of a finite number of Fourier coefficients (with respect to the space variable) of the solutions of the system \eqref{eq:BrusselatorInfODE}. The influence from the omitted variables, i.e. the ones from $Q$, is estimated from the  enclosure obtained in the first step.
        The differential inclusion is integrated rigorously over time interval $\tau$ and for initial values belonging to the set $X_P(0)$. As a result we obtain the bounds for the coordinates of $X_P(\tau)$. There are two possibilities to get these bounds:
        \begin{itemize}
            \item Study the evolution of variables belonging  to $P$ separately, coordinate by coordinate, by solving the linear differential inequalities. That is, for every $i\in\mathbb{N}$ we are estimating the evolution of the Fourier modes from the equation
            \begin{equation*}
                \frac{d}{dt}x_i(t) = \lambda_i x_i(t) + f_i(x(t)).
            \end{equation*}
            \item Solve rigorously the following vector differential inclusion, obtained by considering the Galerkin projection of the problem \eqref{eq:AbstractProblem} on the space $H_P,$ and estimating the influence the omitted terms in the equation through the multivalued expression $I$.

            \begin{equation*}
                \frac{d}{dt}Px \in PF(Px) + I.
            \end{equation*}

            The  rigorous integration algorithm for finite dimensional vector inclusions such as the one above is described in \cite{ControlKapelaZgliczynski}.
        \end{itemize}
        We  intersect the estimates obtained by two above techniques in order to obtain the sharper bounds.

        \item Use the a priori estimates coming from the dissipativity of the linear part of the problem to obtain the representation of $X_Q(\tau)$. The influence of the nonlinear terms is estimated from the enclosure found in the first step.
    \end{itemize}
    The details of this step of the algorithm are given in Section
    \ref{sec:Evolving set}.
\end{enumerate}
\subsection{Representation of sets}\label{sec:descritionSet}
The important role in the algorithm will be  played by the sequences $\{V_i\}_{i=1}^\infty$ of intervals which we will call infinite interval vectors. For a given infinite interval vector $V$ we will denote the $i$-th interval by $V_i$ and its left and right ends by
$V_i^-,V_i^+$, respectively.
If for every sequence  $\{v_i\}_{i=1}^\infty$, such that $v_i\in V_i$, the
series
$\sum_{i=1}^\infty e_iv_i$ converges in $H$, we will call the set $\{\sum_{i=1}^\infty e_iv_i \in H : v_i\in V_i \}$ a representation of infinite interval vector and we will say that the infinite vector is representable in the space $H.$
It is possible that infinite interval vector $V$ does not represent subset of $H$ as it can happen that the series
$\sum_{i=1}^\infty e_iv_i$ where $v_i\in V_i$ does not converge in $H$.
 Whenever it will not lead to confusion we will use the same nonantion for infinite interval vectors and their representations.

We define several useful operation on the infinite interval vectors. First, for infinite vector  $V$ we
define the quantities
\begin{equation*}
    V^- = \{[V^-_i,V^-_i]\}_{i=1}^\infty,
    \quad
    V^+ = \{[V^+_i,V^+_i]\}_{i=1}^\infty.
\end{equation*}
For a given interval $I$ the we define multiplication an infinite interval vector $V$ by the interval $I$ as
\begin{equation*}
I\, V= V\, I =  \{I V_i\}_{i=1}^\infty.
\end{equation*}
For two infinite intervals vectors
$V$ and
$W$ we define their sum and element-wise product as
\begin{equation*}
    V+W = \{V_i + W_i\}_{i=1}^\infty,\quad V*W = \{V_iW_i\}_{i=1}^\infty.
\end{equation*}
We say that vector $V$ is a subset of $W$ and denote by
$V\subset W$ if and only if
$V_i\subset W_i$ for every $i\in\mathbb{N}.$
Additionally we define
$V\subset_\text{int} W$ if and only if
$V_i\subset \text{int} W_i$ for every $i\in\mathbb{N}.$ We define the convex hull of two infinite intervals vectors as
\begin{equation*}
  \text{conv}\{V,W\} =\{ \text{conv}\{V_i\cup W_i\} \}_{i=1}^\infty.
\end{equation*}
The intersection of two infinite vectors is defined in the following way
\begin{equation*}
  V \cap W =\{ V_i\cap W_i \}_{i=1}^\infty.
\end{equation*}
Note that all above operations make sense for all infinite interval vectors, and not only the representable ones.

In the algorithm we consider such sets $X = X_P + X_Q$, for which there exist infinite interval vectors whose representation contains the set $X.$
Specifically, for the Brusselator system we will work with the pairs of interval infinite vectors which are given in the form $(U,V) =\{(U_i,V_i)\}_{i\in\mathbb{N}^+}$ where $U_i$ and $V_i$ are the intervals. Such vectors  can be easily re-indexed into the form described previously. We will work with class of infinite interval vectors, with polynomial estimates on the tail. That means that for some $n\in\mathbb{N}$ and $s\in\mathbb{R}$ and every sequence $u_i\in U_i$ and $v_i\in V_i$ we have
\begin{equation}\label{eq:representationInfiniteVectors}
    u_i\in \frac{[C_U^-,C_U^+]}{k^s},\quad
    v_i\in \frac{[C_V^-,C_V^+]}{k^s},\quad
    \text{for $i\geq n$},
\end{equation}
where $C_U^-\leq C_U^+$ and $C_V^-\leq C_V^+$ are constants. In this manner, the tail of the Fourier expansion (for $i\geq n$) can be represented by specifying the decay rate $s$ of the coefficients and four additional constants $C_U^-, C_U^+, C_V^-, C_V^+$.
Lemmas \ref{lem:add} and \ref{lem:mult} are helpful in the  implementation of operations of element-wise multiplication and  addition for such class of infinite interval vectors.

Since for the Bruselator, we consider the system of two PDEs, the states are the sets $X = X_P+ X_Q$ which are the subsets of $Y$. Their elements are pairs $(u,v).$ To represent these pairs we use
the Fourier basis which consists of the eigenfunctions of the operator $L$, i.e. $(\sin(kx),0)$ and $(0,\sin(kx))$ for  $k\in \mathbb{N}^+$. The finite dimensional space  $H_P,$ in which the sets $X_P$ are always contained, is equal to
$\text{span}\{(\sin(x),0),(0,\sin(x)),\ldots ,(\sin(kx),0),(0,\sin(kx)) \}.$

The set $X$ has to be a subset of some representable infinite interval vector vector $(U,V)$.
In the algorithm, this means that $X_P$ is a subset of some cube in $H_P.$
The cubes are the simplest examples of possible representations of the set $X_P.$
More sophisticated parallelepiped-type objects can also used. They can reduce overestimation of the integration results for the rigorous ODE solvers and overcome so calling wrapping effect  (for example see \cite{CAPDArticle}).
The element $(u,v) \in H_Q$ belongs to $X_Q$ if it satisfies \eqref{eq:representationInfiniteVectors} with $s>1.$ This guarantees that $X_Q$ is a subset of $Y$. In that case we can estimate the result of series multiplication by using Lemmas \ref{lem:sinTimesSin} and \ref{lem:cosTimesSin}.
  Finally, we note that we may restrict our space to the space of functions $u,v$ with only odd nonzero coefficients. Then the set $X$ is a subset of $W := \{(u,v)\in Y: u_{2i} = v_{2i} = 0, \;\text{for}\;i\in\mathbb{N}\}.$
   The representation of set $X$ is roughly the same, except that we have to enforce that $u_i=v_i=0$ if $i$ is even.

\subsection{Computation of nonlinear terms.}\label{sect24} In the course of the algorithm for a given set $X=X_P+ X_Q$, we need to compute the set $X^1$ such that $ f(X)\subset X^1.$
This set, represented as $X^1 = X_P^1+ X_Q^1$  constitutes the estimates for $f(X)$ and hence it should be as small as  possible.
The set $X^1$ is used in further steps of algorithm. For the Brusselator problem we have that
\begin{equation*}
    f(u,v) = (u^2v + A\sin(x),Bu-u^2v).
\end{equation*}
For functions $(u,v)\in Y$ the components of $f$ can be represented in the following sine Fourier series with the coefficients $a_i, b_i$ dependent on $u, v$
\begin{equation*}
    u^2v + A\sin(x) = \sum_{i=1}^\infty a_i \sin(i x),\quad Bu-u^2v = \sum_{i=1}^\infty b_i \sin(i x).
\end{equation*}
The set $X^1_P$ is represented as the cube (or parallelogram) in $H_P$ and $X^1_Q$  is  described by the polynomial decay of  Fourier coefficients.
The first step of finding $X^1$ is estimating the square of $u$ which is represented in the cosine Fourier series. For this purpose we use Lemma  \ref{lem:sinTimesSin}.
Having computed the coefficients of $u^2$ we need to find the coefficients of $(u^2)v.$ To this end we use Lemma \ref{lem:cosTimesSin}. Finally, we use Lemma
\ref{lem:add} to compute the representation of sums of particular terms which appear in the definition of $f$.
We also need to compute the image $L(X)$ but as $L$ is a diagonal operator we only need to multiply every given coefficient by the corresponding eigenvalue of $L.$ The result of such multiplication is given in Lemma \ref{lem:mult}.
Additionally in the algorithm we  need the decomposition
 \begin{equation*}
     f(p+q) = f(p) +f_2(p,q),
 \end{equation*}
 where $p\in H_P$, $q\in H_Q$ and $f_2(p,q) = f(p+q)-f(p)$.  This decomposition is required for the formulation of the differential inclusion.
For the Brusselator system we can write
$f(u_P+u_Q,v_P+v_Q) = f(u_P,v_P) + f_2(u_P,v_P,u_Q,v_Q)$ where
%\begin{equation*}
%    f_1(u_P,v_P) = (u_P^2v_P + A\sin(x),Bu_P -  u_P^2v_P),
%\end{equation*}
\begin{equation*}
    f_2(u_P,v_P,u_Q,v_Q) = (
    (2u_Pu_Q+u_Q^2)(v_P+v_Q)+ u_P^2v_Q,
    Bu_Q-(2u_Pu_Q+u_Q^2)(v_P+v_Q)- u_P^2v_Q
    ).
\end{equation*}

\subsection{Computation of  the enclosure.}\label{sec:enclosure} We start this section with the definition of a enclosure.
\begin{definition}
The set $X([0,\tau])$ is a enclosure of the set $X^0 \subset Y$ for time $\tau > 0$ if $\varphi(t,X^0) \subset X([0,\tau])$ for every $t\in [0,\tau]$.
\end{definition}
The following Lemma can be used in order to validate if for the given set of initial data $X^0$ the set $X^0+Z$ is an enclosure.
\begin{lemma}\label{lemma:enclosure}
Assume that (A1) and (A2) hold and let $\{X^0_i\}_{i=1}^\infty$ be a countable family of intervals $X^0_i = [x_i^-,x_i^+]$ such that the set $X^0:= \set{\sum_{i=1}^\infty e_i x_i: x_i\in X_i^0}$ is bounded in $Y$. Moreover, define another set
$Z:= \set{\sum_{i=1}^\infty e_i z_i: z_i\in Z_i},$ also bounded in $Y$, where $Z_i = [z^-_i,z^+_i]$ are intervals containing zero.
 Let $x^0\in X^0$ and $n\in\mathbb{N}$ . We assume that for every $i\in \mathbb{N}$ there holds
\begin{equation}\label{eq:EnclosureCondition}
    [g_i^-,g_i^+]\cap
    [h_i^-,h_i^+]\subset \text{int } Z_i,
\end{equation}
where
\begin{equation}\label{eq:gminusgplus}
        g^-_i =
        \min_{t\in[0,\tau]}
        \left[(e^{\lambda_i t} - 1)
        \left(\frac{f^-_i}{\lambda_i} +x^-_i\right)\right],
        \quad
        g^+_i =
        \max_{t\in[0,\tau]}
        \left[(e^{\lambda_i t} - 1)
        \left(\frac{f^+_i}{\lambda_i} +x^+_i\right)\right],
    \end{equation}
\begin{equation}\label{eq:hminushplus}
    h^-_i = \min_{t\in[0,\tau]}
        t(\lambda_i(x_i^-+z_i^-) +f_i^-) ,
        \quad
        h^+_i =
        \max_{t\in[0,\tau]}
        t(\lambda_i(x_i^+ + z_i^+) +f_i^+) ,
\end{equation}
     and $f^-_i,f^+_i\in\mathbb{R}$  satisfy
    \begin{equation}\label{eq:fpm}
        f^-_i\leq f_i(X^0+Z)\leq f^+_i.
    \end{equation}
Then for every $x^0\in X^0$ there exists a continuous function $x:[0,\tau]\to Y$
which is a unique solution to $\eqref{eq:AbstractProblem}.$ Moreover for every $t\in[0,\tau]$ we have
\begin{enumerate}
    \item  $x_i(t)\in x^0_i+[0,t](f_i(X^0+Z)+\lambda_i(X^0+Z))$ for every $i\in \mathbb{N},$
    \item $x^-_i+g_i^-\leq x_i(t)\leq x^+_i+g_i^+$ for every $i\in \mathbb{N},$
    \item $x_i(t)\in e^{t \lambda_i} x^0_i + \frac{e^{\lambda_i t} -1}{\lambda_i}[f^-_i,f^+_i]$ for every $i\in \mathbb{N}.$
\end{enumerate}
\end{lemma}
\begin{proof}
We define the operator $T:C([0,\tau];Y)\to C([0,\tau];Y) $ in the following way
\begin{equation}
    T(g)(t) = e^{Lt}x^0 + \int_0^t e^{L(t-s)}f(g(s))ds.
\end{equation}
We consider the set
\begin{equation*}
    S_\tau = \{g\in C([0,\tau];Y):g(0)=x^0\; \text{and for every $t\in[0,\tau]$ we have } g(t)\in X^0+ Z\}.
\end{equation*}
We prove first that for every $x^0\in X^0$ the mapping $T$ leads from $S_\tau$ to itself.
We observe that $y_i(t) = T_i(g)(t)$ is a solution to the non-autonomous ODE

\begin{equation}
    \frac{d}{dt}y_i(t) = \lambda_iy_i(t) + f_i(g(t)),
\end{equation}
for every $i\in\mathbb{N.}$
We will prove that $y_i(t) \in [x^-_i+z^-_i,x^+_i +z^+_i ],$ for every $t\in[0,\tau].$ For the sake of contradiction assume that $\tau_1 < \tau,$ where
$\tau_1 := \sup_{t\in[0,\tau]}\{y_i(s)\in [x^-_i+z^-_i,x^+_i +z^+_i ]\ \text{for every } s\leq t \}.$ We observe that
\begin{equation*}
    y_i(t) = x_0+ \int_0^t\lambda y_i(s)+f_i(g(s))ds,\quad
    y_i(t) = e^{\lambda_i t} x_0+ \int_0^t e^{\lambda(t-s)}f_i(g(s))ds.
\end{equation*}
So for $t\leq \tau_1$ we have
\begin{equation*}
    y_i(t)\leq x^0_i+\int_0^t(\lambda (x^+_i+z^+_i)+f_i^+)ds\leq x^+_i + h_i^+.
\end{equation*}
Similarly for $t\leq \tau_1$ we have  $x^-_i + h_i^-\leq y_i(t).$
We observe that for $t\leq \tau_1$
\begin{equation*}
y_i(t)\leq e^{\lambda_i t}x^0_i+ \int_0^t e^{\lambda_i(t-s)}f_i^+ds\leq
e^{\lambda_i t}x^+_i + \frac{e^{\lambda_i t}-1}{\lambda_i}f_i^+ = x^+_i + g_i^+.
\end{equation*}
We also see that $x^-_i + g_i^-\leq y_i(t).$ This implies that for every
$t\in [0,\tau_1]$ we have $x^-_i+z^-_i <y_i(t)<x^+_i +z^+_i.$ By the continuity of $y_i(t)$ we can find $\delta>0$ such that $y_i(t)\in [x^-_i+z^-_i,x^+_i +z^+_i ]$ for every $t\leq\tau_1+\delta.$
This is a contradiction, so $\tau=\tau_1$. We deduce that $T(S_\tau)\subset S_\tau.$
 To show that $T$ is a contraction  we equip the space  $C([0,\tau],Y)$ with the  norm $\|y\|_{\alpha} = \sup_{t\in[0,\tau]} \norm{y(t)}_Y e^{-t\alpha},$ where $\alpha>0$ is an appropriately chosen positive constant. For $g_1,g_2\in S_\tau$ we have
\begin{align*}
     \left\|\int_0^t e^{L(t-s) }( f(g_2(s)) -f(g_1(s)))ds\right\|_Y
     &\leq e^{t C_1} \int_0^te^{s\alpha } e^{-s\alpha} \norm { f(g_1(s)) -f(g_2(s)))}_Y  ds\\
      &\leq
      \frac{e^{t C_1}  e^{t\alpha }}{ \alpha}
      \sup_{s\in[0,t]} \left(e^{-s\alpha}  \norm { f(g_1(s)) -f(g_2(s))}_Y\right)
      \\
      &\leq \frac{C(R) e^{t C_1}  e^{t\alpha }}{ \alpha}
      \sup_{s\in[0,t]} e^{-s\alpha}  \norm { g_1(s) -g_2(s)}_Y.
\end{align*}
If we pick
$\alpha = \sup_{s\in[0,\tau]} \frac{1}{2 C(R) e^{s C_1}  }$ then
\begin{equation*}
    \left\|T(g_1) - T(g_2) \right\|_\alpha \leq \sup_{s\in[0,\tau]} \frac{C(R)e^{sC_1} }{\alpha}\norm{g_1 - g_2 }_{\alpha}\leq\frac{1}{2}\norm{g_1 - g_2 }_{\alpha}.
\end{equation*}
where $R>0$ is a bound in the norm $\norm{.}_Y$ of elements of the set $X_0+Z.$ This shows that $T$ is a contraction. From the Banach fixed point theorem, we deduce that $T$ must have a unique fixed point in $S_\tau$.  Now, formulae (1)-(3) are straightforward, which completes the proof.
\end{proof}
\begin{remark}
The assertion of  Lemma \ref{lemma:enclosure} holds if we replace conditions $(A1)-(A2)$ by their more general counterparts $(B1)-(B3)$.
\end{remark}
\begin{proof}
We consider the same operator $T$ and the set $S_\tau$ as in proof of Lemma \ref{lemma:enclosure}. The argument that $T(S_\tau)\subset S_\tau $ is the same, only the proof of contractivity changes. We will show that in this more general case we can find $\alpha>0$ such that $T$ is a contraction with respect to the norm $\norm{y}_{\alpha} = \sup_{t\in [0,\tau]} e^{-t\alpha}\norm{y(t)}_{Y}.$
For $g_1,g_2\in S_\tau$ we compute
\begin{align*}
     \left\|\int_0^t e^{L(t-s) }( f(g_2(s)) -f(g_1(s)))ds\right\|_Y
     &\leq C_3 e^{t C_4} \int_0^t \frac{1}{(t-s)^{\gamma}}
     e^{s\alpha } e^{-s\alpha} \norm { f(g_1(s)) -f(g_2(s)))}_{Y^1}  ds\\
      &\leq
      C_3e^{\tau C_4}C(R)
  \norm { g_1 -g_2}_{\alpha}
      \int_0^t\frac{1}{(t-s)^\gamma}e^{\alpha s} ds
      \\
      &=
      C_3e^{\tau C_4}C(R)e^{t\alpha}\norm{g_1-g_2}_{\alpha}
      \int_0^t\frac{1}{s^\gamma}e^{-s\alpha} ds
      .
\end{align*}
where $R>0$ is a bound of the norm $\norm{
\cdot
}_Y$ satisfied by all elements of the set $X_0+Z.$ For every $\delta\leq t$ we have
\begin{equation*}
    \int_0^t\frac{1}{s^\gamma}e^{-s\alpha} ds\leq \int_0^\delta
    \frac{1}{s^\gamma}ds+\frac{1}{\delta^\gamma}\int_\delta^te^{-s\alpha}ds\leq
    \frac{\delta^{1-\gamma}}{1-\gamma} +\frac{1}{\delta^\gamma}\frac{1}{\alpha}
\end{equation*}
So, if we pick $\delta,\alpha$ such that
\begin{equation*}
    \delta^{1-\gamma}\leq\frac{1}{4(1-\gamma)}C_3 e^{\tau C_4} C(R),\quad \alpha \geq \frac{1}{4\delta^\gamma}C_3e^{\tau C_4}C(R),
\end{equation*}
then for every $t\leq \tau$ we have
\begin{equation*}
    \left\|\int_0^t e^{L(t-s) }( f(g_2(s)) -f(g_1(s)))ds\right\|_Y \leq \frac{1}{2}e^{t\alpha}\norm{g_1-g_2}_{\alpha}.
\end{equation*}
Hence, $T$ is a contraction on the set $S_\delta$ equipped with the norm $\norm{\cdot}_\alpha$. The Banach fixed point theorem gives the unique fixed point of the map $T$, and the proof is complete.
\end{proof}
   In Lemma \ref{lemma:enclosure} the sets $X^0,Z$ are
  representable infinite interval vectors.
 If $X^0$  is not a representable infinite interval vector, we can still find a enclosure, provided we can choose the representable interval vector $X^1$ such that $X^0 \subset X^1.$
 The following algorithm takes as an input the infinite interval vector $X^0,$ the infinite interval vector $Z$ such   $0\in Z_i$ for every $i\in\mathbb{N},$  and the step $\tau>0.$ The algorithm  validates if $X^0 + Z$ is an enclosure for $\tau$ by checking assumption \eqref{eq:EnclosureCondition} of Lemma  \ref{lemma:enclosure}.

 \begin{enumerate}
     \item Find the infinite vectors $V^f$ and $V^L$
     such that $f(X+Z) \subset V^f $ and
     $L(X+Z)\subset V^L.$
     \item Construct the infinite vector $V^{\text{ND}}$ such that
     $[0,\tau](V^L+V^f)\subset V^{\text{ND}} $
    \item Compute the infinite vectors $V^{E_1}$ and $V^{L^{-1}}$ which satisfy
    \begin{equation*}
        e^{[0,\tau]\lambda_i} - 1 \subset (V^{E_1})_i,
        \qquad
        \frac{1}{\lambda_i} \in (V^{L^{-1}})_i \ \ \textrm{for every}\ \ i\in \{1,\dots\}.
    \end{equation*}
     \item Compute the infinite vectors $V^{\text{D1}}$ and $V^{\text{D2}}$ which satisfy
     \begin{equation*}
         V^{E_1}*((V^f)^+*V^{L^{-1}} + X^+  ) \subset V^{\text{D1}},
         \quad
         V^{E_1}*((V^f)^-*V^{L^{-1}} + X^- ) \subset V^{\text{D2}}.
     \end{equation*}
     \item Find infinite interval vectors  $V^{\text{D}}$ such that
     \begin{equation*}
        \text{conv}\{V^{\text{D1}}, V^{\text{D2}} \} \subset V^{\text{D}}.
     \end{equation*}
     \item Compute $Z^1$ such that  $V^\text{ND}\cap V^\text{D}\subset Z^1.$
     \item Check if $Z^1 \subset_{\text{int}} Z$ holds.

 \end{enumerate}
 For every $i\in\mathbb{N}$ the interval $[h_i^-,h_i^+]$ from  \eqref{eq:hminushplus} is contained in the interval $V^{ND}_i,$ where $V^{ND}$ is an infinite interval vector obtained in step (2).
 Similarly for every $i\in\mathbb{N}$ the interval $[g_i^-,g_i^+]$ from \eqref{eq:gminusgplus}  is contained in the interval $V^{D}_i,$ where $V^{D}$ is an
 infinite interval vector from step (5). If condition in step (7) is satisfied, we have that $V^{ND}_i\cap V^{ND}_i \subset \text{int} Z_i$ for every $i\in\mathbb{N},$ which implies that assumptions of Lemma \ref{lemma:enclosure} hold and the set $X^0 + Z$ is an enclosure for our initial data $X^0$ and the time step $\tau>0$ .

 If condition $(7)$ does not hold, we can modify the set $Z$ and repeat the validation procedure. The reasonable guess is to take $Z = [0,c]Z^1,$ where $c>1.$ Another possibility is to decrease the time step. The following lemma shows that for certain class of sets  it is always possible to find the enclosure using the above algorithm with sufficiently small time step $\tau$.

 \begin{lemma}
 \label{lem:whyItWorks}
Let $\{a_i\}_{i=1}^\infty$ be a given sequence of positive numbers. Assume that
\begin{itemize}
    \item[(I)] Conditions (A1)-(A2) or (B1)-(B3) hold.
    \item[(II)] For some $i_0\in\mathbb{N}$ we have $\lambda_i < 0$ for every $i\geq i_0.$
    \item[(III)] For every bounded set $A\subset Y$ such that
    \begin{equation*}
        \sup_{x\in A} |x_i| = O(a_i),
    \end{equation*}
    the following holds
    \begin{equation*}
        \sup_{x\in A} \left|\frac{f_i(x)}{\lambda_i}\right| = o(a_i).
    \end{equation*}
\end{itemize}
 Let  $\{X^0_i\}_{i=1}^\infty$ and $\{Z_i\}_{i=1}^\infty$ be sequences of intervals such that
\begin{itemize}
    \item[(IV)] The set $X^0:= \set{\sum_{i=1}^\infty e_i x_i: x_i\in X_i^0}$ is bounded in $Y$. Intervals $X_i^0$ contain zero for
    every $i\geq i_0,$ where $i_0\in\mathbb{N}$
    and
    \begin{equation*}
         \quad \sup_{x_i\in X_i^0}|x_i| = O(a_i).
    \end{equation*}
    \item[(V)] The set $Z:= \set{\sum_{i=1}^\infty e_i z_i: z_i\in Z_i}$ is bounded in $Y$. Every interval $Z_i$ contains zero and
    \begin{equation*}
        z_i^+ = \Theta(a_i)
        \quad \text{and}\quad
        z_i^- = \Theta(a_i).
    \end{equation*}
\end{itemize}
Under these assumptions there exists $\tau>0$ such that condition \eqref{eq:EnclosureCondition} of Lemma \ref{lemma:enclosure} is satisfied, and, in consequence, for every $x^0\in X^0$ there exists a continuous function $x:[0,\tau] \to Y$ which is a unique solution to \eqref{eq:AbstractProblem} and the following estimates hold
for every $t\in[0,\tau]$
\begin{enumerate}
	\item  $x_i(t)\in x^0_i+[0,t](f_i(X^0+Z)+\lambda_i(X^0+Z))$ for every $i\in \mathbb{N},$
	\item $x^-_i+g_i^-\leq x_i(t)\leq x^+_i+g_i^+$ for every $i\in \mathbb{N},$
	\item $x_i(t)\in e^{t \lambda_i} x^0_i + \frac{e^{\lambda_i t} -1}{\lambda_i}[f^-_i,f^+_i]$ for every $i\in \mathbb{N},$
\end{enumerate}
where $g_i^-, g_i^+$ are given by \eqref{eq:gminusgplus} and $f_i^-, f_i^+$ are given by \eqref{eq:fpm}.  
\end{lemma}
\begin{proof}
Observe that we can find $i_1$ such that
\begin{equation}\label{enclosure_1}
   [g_i^-,g_i^+] \subset \text{int}\; Z_i = (z_i^-,z_i^+),
\end{equation}
for every $i\geq i_1$ and $\tau>0.$
Indeed for big enough $i$ we have
\begin{equation*}
    g_i^+=\max_{t\in[0,\tau]}
        \left[(e^{\lambda_i t} - 1)
        \left(\frac{f^+_i}{\lambda_i} +x^+_i\right)\right]
        \leq
         \max_{t\in[0,\tau]}
        \left[(e^{\lambda_i t} - 1)
        \frac{f^+_i}{\lambda_i}\right] \leq \left|\frac{f_i^+}{\lambda_i}\right|< z_i^+.
\end{equation*}
The first inequality follows from the fact that $x_i^+$ is positive and term $(e^{\lambda_i t} - 1)$ is negative for sufficiently large $i$. The second inequality follows from the assumption on $\lambda_i.$ The third one is a consequence of the fact that $\sup_{x\in X} \left|\frac{f_i(x)}{\lambda_i}\right| = o(a_i)$ and $z^+_i = \Theta(a_i).$
%So we observe that $g^+_i<z_i^+.$
%Similarly for $i$ big enough we have $g^-_i >z_i^-.$
We have shown that for $i\geq i_1$ we have the inclusion \eqref{enclosure_1}. For the remaining, ''low'', indexes $i$,  we take  $\tau>0$ such that
\begin{equation*}
    \tau < \sup_{i<i_1}\frac{|z_i^-|}{ \left|(\lambda_i(x_i^-+z_i^-) +f_i^-)\right| + 1}
    \quad \text{and}\quad
    \tau < \sup_{i<i_1} \frac{|z_i^+|}{\left|(\lambda_i(x_i^++z_i^+) +f_i^+)\right| + 1}.
\end{equation*}
 For such choice of $\tau$ we have
 \begin{equation*}
     [h_i^-,h_i^+]\subset \text{int}\;Z_i\quad\text{for}\quad i< i_1,
 \end{equation*}
 which ends the proof.
\end{proof}
\begin{remark}
	The key assumption of the above lemma is (III). The sequence $a_i$ signifies some decay of Fourier coefficients of the elements $x$ of a set $A$. The decay of Fourier coefficients of $\frac{f_i(x)}{\lambda_i}$ for $x\in A$ must be essentially faster than $a_i$. Assume for example that $a_i = \frac{1}{i^s}$ and that $f$ is a polynomial. Then, the decay of the coefficients of $f_i(x)$ is the same as that of $x$, that is also $\frac{1}{i^s}$ (see Section \ref{sec:algebra}). If the leading operator is dissipative, i.e. $\lambda_i \to -\infty$, then, after the division by $\lambda_i$ this decay will be essentially faster then $a_i$. Thus, dissipativity of $L$ guarantees that the enclosure can be always found, and a step of the rigorous integration algorithm can be performed. The drawback is, that the length of a time-step $\tau$ can be very small in Lemma \ref{lem:whyItWorks}. We expect that such situation holds for dissipative problems with finite type blow-up, for instance for the problems governed by the Fujita equation $u_t=u_{xx}+|u|^{p-1}u$, cf. \cite{Fujita}.   
\end{remark}
Using the above lemma we deduce that for the Brusselator system we can always find an enclousure for any set described in the Section \ref{sec:descritionSet}, by choosing sufficiently short time-step.

  \begin{remark}
 For every $C>0,\;\varepsilon>0,\;k>1$ there exists $\tau>0$ such that for every initial data $(u^0,v^0)\in Y$ satisfying
 \begin{equation*}
     |u^0_k|\leq\frac{C}{|k|^s}
     \quad\text{and}\quad
     |v^0_k|\leq\frac{C}{|k|^s},
 \end{equation*}
 the solution $(u(t),v(t))$ to \eqref{eq:BrusselatorPDE} satisfies
 \begin{equation*}
     |u_k(t)|\leq\frac{C+\varepsilon}{|k|^s}
     \quad\text{and} \quad
     |v_k(t)|\leq\frac{C+\varepsilon}{|k|^s}
     \quad\text{for}\;t\in[0,\tau].
 \end{equation*}

 \end{remark}
 One can easily construct an example of a problem without this property. To this end let us consider the logistic model of population growth with diffusion and homogeneous Dirichlet boundary conditions. The corresponding equation together with the initial and boundary conditions are the following
\begin{equation}\label{eq:logisticEquation}
\begin{cases}
 u_t = d  u_{xx} + u- u^2\   \text{for}\; (x,t)\in [0,\pi]\times(0,\infty),\\
u(t,x) = 0\; \text{for}\; (x,t)\in \{0,\pi\}\times (0,\infty),
\\u(0,x) = u^0(x),
\end{cases}
\end{equation}
 for $d>0.$ As we impose the Dirichlet boundary condition, from the equation we observe that $u_{xx}$ is always $0$ at the boundary. But, in general, the fourth derivative $u_{xxxx}$ can have nonzero values at the boundary.  This implies that the coefficients in the expansion of the solution in the sine trigonometric series  cannot have the arbitrarily fast polynomial decay.

 We can write the equation  in above example in the form \eqref{eq:AbstractProblem} with $L = d u_{xx}+u$ and $f(u)=u^2.$ For the initial condition $u_0 = \sqrt{\frac{\pi}{8}}\sin(x)$ we have that
 \begin{equation*}
      f(u_0) = \frac{\pi}{8}(1-\cos(2x)) = \sum_{i=1}^\infty \frac{(-1+(-1)^{i} )}{-4n+n^3}\sin(ix).
 \end{equation*}
So using Lemma \ref{lemma:enclosure} we have a chance to find the time step $\tau>0$ and $C>0$ such that the solution satisfies
\begin{equation*}
    |u_k(t)|\leq \frac{C}{k^s}\quad \text{for}\quad t\in[0,\tau],
\end{equation*}
 only for $s\in(1,5).$ Similar difficulty would occur if in the considered Brusselator system the term $\sin(x)$ would be replaced with the original term from the Brusselator ODE,  that is the constant $1$.

 \subsection{Evolution of sets}\label{sec:Evolving set}
 We assume that $X([0,\tau]) = X_P([0,\tau])+ X_Q([0,\tau])$ is an enclosure for the initial data $X = X_P+ X_Q$ and the time step $\tau>0.$  The following procedure is used to find the set $X(\tau)$ such that $S(\tau)X\subset X(\tau).$

\begin{enumerate}
    \item Compute the infinite interval vector $V$ such that
    $f_2(X_P([0,\tau]),X_Q([0,\tau]))\subset V$.
    \item Solve the system of differential inclusions
    \begin{equation}\label{inclusion}
        \frac{d}{dt}Px \in L Px+Pf(Px) + PV,
    \end{equation}
    with initial condition in the set $X_P$. Note that if the initial data belongs to the set $X_P$ which has more complicated structure then a vector of intervals (for example it can be a parallelepiped) this step, can be realized without enclosing the initial data in an interval vector. This will result in sharper estimates.
    As the result, the solver will generate the set $X_{P1}$ such that $Px(\tau) \in X_{P1}$ for every $x(0) \in X_P+X_Q$.
    \item Compute the infinite interval vector $V^2$ such that $f(X([0,\tau])) \subset V^2.$
    % \item Find interval vector $V_3$ such that $X\subset V_3.$
        \item Compute the infinite vectors $V^{E_1}, V^{E_2}$ and $V^{L^{-1}}$ which satisfy
    \begin{equation*}
        e^{\tau\lambda_i} -1 \in (V^{E_1})_i,
        \qquad
        e^{\tau\lambda_i}  \in (V^{E_2})_i,
        \qquad
        \frac{1}{\lambda_i} \in (V^{L^{-1}})_i \ \ \textrm{for every}\ \ i\in \{1,\dots\}.
    \end{equation*}

    \item Compute $V^3$ such that
        $V^{E_2} * X + V^{E_1} * V^{L^{-1}} * V^2 \subset V_3$.

    \item Return the set $(PV_3 \cap X_{P1}) + QV_3$.
\end{enumerate}

Note that the interval vectors  $V^{E_1}$ both in the algorithm for finding the enclosure and evolving the set are not representable. However, we still need the data structures to represent them as infinite vectors and perform the operations such as addition or elementwise multiplication. Particular way of representing such sets and their operations with polynomial decay or growth of coefficients  is addressed in Section \ref{sec:algebra}.

The detailed description how to rigorously solve the differential inclusion can be found in \cite{ControlKapelaZgliczynski}.
For the step (2) we can consider differential inclusion only on part of variables represented explicitly.
This can be beneficial for the computational time.

As we are using the infinite interval vectors \eqref{eq:representationInfiniteVectors} we need to determine the decay rate $s$ for $V^{E_2}.$ It is possible to impose arbitrarily fast polynomial decay in this term. On the other hand, the vector $V^{E_1}$ can be represented by \eqref{eq:representationInfiniteVectors} with $s=0$. Finally, the representation of $V^{L^{-1}}$ decays with some given $s_1>0$ determined by the decay of the inverses of the eigenvalues of $L$ (which have to decay to zero as the considered problem is disspative). So, $V_3$ can have higher $s$ then the initial data $X$. Maximal increase of the rate $s$ in $V_3$ is equal to $s_1$. Theoretically in every time step it is possible to increase the decay $s$ by the value of $s_1.$ But it can lead to overestimates on some variables so it is sometimes beneficial to keep the old $s.$ In the code we are using heuristic algorithm which estimates upper bound of resulting series to decide if it worth to increase  the exponent.

\section{Computer assisted proof of periodic orbit existence}\label{sec:comAssisedProof}
In this section we describe the computer assisted proof of Theorem \ref{th:BrusselatorPeriodicOrbit}.
\subsection{Overview of the proof}
The proof of Theorem \ref{th:BrusselatorPeriodicOrbit}
is based on the Schauder fixed point theorem. We check that for some previously prepared Poincar\'{e} map $\mathcal{P}$ the image of appropriately chosen compact initial set $X^0$ is contained in itself, i.e. $\mathcal{P}(X^0)\subset X^0.$
To validate the inclusion, we apply the previously described integration algorithm and Lemma \ref{lem:cross} to address the issue of crossing the section. As the preliminary step,
we construct a set of initial data $X^0$ and a section $l$ from the analysis of results of approximate numerical integration of the Galerkin projection both of the Brusselator equation and its variational equation.

\subsection{Poincar\'{e} map - crossing the section}\label{sec:crossing}
The problem of finding the periodic orbit for the system \eqref{eq:AbstractProblem} is reduced to finding the fixed point of the Poincar\'{e} map. In this section we present the algorithm of rigorous computatation of the Poincar\'{e} map and justify its correctness. While the same algorithm is  for ODEs can be found in \cite[Section 5]{ZLohner}, and its infinite dimensional version appears in \cite[Section 3]{ZPKuramotoII}, we present it for the exposition completeness. Note that although the results of the present section follow closely the concepts of \cite{ZLohner, ZPKuramotoII}, the results in \cite[Lemma 6 and Theorem 8]{ZPKuramotoII} use the Brouwer fixed point theorem, and the argument on passing to the limit in Galerkin projections to get the fixed point, and our Theorem \ref{th:PoincareFixedPoint} uses the Schauder fixed point theorem. 
 
We assume as in Section \ref{sec:AlgorithmIntegration} that we consider abstract  problem \eqref{eq:AbstractProblem}. The spaces $H,Y$ are also the same as in Section \ref{sec:AlgorithmIntegration}. By $\varphi$ we define the local semiflow which is given by solutions of
\eqref{eq:AbstractProblem}.
The following lemma allows us to check if evolution of initial set transversally intersects the section, which is a kernel of some affine map.
\begin{lemma}\label{lem:cross}
Let $l:Y\to \mathbb{R} $ be a function given by the formula
$l(x) = \sum_{i=1}^n \alpha_i ( x_i -\beta_i) $, where $\alpha_1\ldots,\alpha_n,\beta_1,\ldots,\beta_n\in \mathbb{R}.$
Let $X$ be a bounded set in $Y.$ We assume that for some $\tau>0$ the following conditions hold:
\begin{enumerate}
    \item For every $x\in X$ we have $l(x)<0\; \text{and}\; l(\varphi(\tau,x))>0.$
\item For every $x\in \varphi([0,\tau],X)$ we have $\sum_{i=1}^n\alpha_i F_i(x) > 0.$
\end{enumerate}
Then for every $x\in X$ there exists a unique $\tau_l(x)\in(0,\tau)$ such that $l(\varphi(\tau_l(x),x))=0.$ Moreover the function
$\tau_l:X\to (0,\tau)$ is continuous in the norm of the space $Y.$
\end{lemma}
\begin{proof}
Let $x\in X.$
We observe that condition (1) and the continuity of the flow imply that there exists $\tau_l(x)$ such that $l(\varphi(\tau_l(x),x)) = 0.$ By condition (2) we deduce that
the function  $g(t) = l(\varphi(t,x))$ is increasing. Indeed we have
\begin{equation*}
    \frac{dg(t)}{dt} =
    \sum_{i=1}^n\alpha_i\frac{d \varphi_i(t,x)}{dt} =
    \sum_{i=1}^n\alpha_i F_i(\varphi(t,x)) > 0 .
\end{equation*}
So $\tau_l(x)$ has to be unique zero for $g$ on the interval $[0,\tau]$. We  prove the continuity of $\tau_l.$ Let $\varepsilon > 0.$ We define $g(\tau_l(x)-\varepsilon) = A_1 <0$ and $g(\tau_l(x)+\varepsilon) = A_2 >0.$ From the continuity of the flow and the function $l,$ we can pick $\delta>0$ such that for every $y\in B_Y(x,\delta) \cap X$ we have
\begin{equation*}
    \sup_{s\in[0,\tau]}
    |l(\varphi(s,x)) - l(\varphi(s,y))| < \frac{\min\{-A_1,A_2\}}{2}.
\end{equation*}
So for $y\in B_Y(x,\delta) \cap X$  we obtain $l(\varphi(\tau_l(x)-\varepsilon,y))<0$ and $l(\varphi(\tau_l(x)+\varepsilon,y))>0.$
Consequently there exists
$\tau_l(y) \in \tau_l(x)+(-\varepsilon,\varepsilon)$ such that
$l(\varphi(\tau_l(y),y))=0,$ so the funtion $\tau_l$ is continuous,  which concludes the proof.
\end{proof}
If assumption of Lemma \ref{lem:cross}  are satisfied then  for every $x\in X$ we have $l(\varphi(\tau_l(x),x)) = 0.$  Hence we can define the map Poincar\'{e} map $\mathcal{P}:X\to Y$ by the formula $\mathcal{P}(x) = \varphi(\tau_l(x),x).$
The map $\mathcal{P}$ is continuous in the norm of the Banach space $Y$.
If $  \varphi(t,X^0)\subset X$ then we can define the Poincar\'{e} map
$\mathcal{P}:X^0\to Y$ by the formula $\mathcal{P}(x)= \varphi(\tau_l(\varphi(t,x)),\varphi(t,x)).$

We will briefly describe the algorithm which estimates the image of the set $X^0 = X^0_P+ X^0_Q$ by the Poincare map.
We assume that the section $l(x)$ is given by the same formula as in Lemma \ref{lem:cross} and does not depend on values in space $H_Q,$ that is for every $x\in Y$ we have that $l(x) = l(Px).$
 The algorithm  is following.
 \begin{enumerate}
     \item  Set $t_\text{prev}:=0$ and $t_\text{curr} :=0.$
     \item Check if for every $x\in \varphi(t_\text{prev},X^0)$ there holds $l(x)<0$ and for every $x\in \varphi(t_\text{curr},X^0)$ there holds $l(x)>0.$
     \begin{enumerate}
         \item If no, then change
         $t_\text{prev}:=t_\text{curr}$
         and
         $t_\text{curr} := t_\text{curr} + \tau,$ where $\tau>0$ is a given time step. Then go back to step
         $(2).$
         \item If yes, then try to minimise the difference $t_\text{curr}- t_\text{prev}$ for which condition in (2) holds.
     \end{enumerate}

     \item Compute the enclosure $X^E$ such that
     $\varphi([t_\text{prev},t_\text{curr}],X^0) \subset X^E.$
     \item Check that for every $x\in X^E$ we have $\sum_{i=1}^n\alpha_i F_i(x)>0$ and if this condition is satisfied return $X^E.$

 \end{enumerate}

     Images $\varphi(t_\text{prev},X^0)$
     and $\varphi(t_\text{curr},X^0)$ can be estimated from the algorithm of integration described in Section \ref{sec:AlgorithmIntegration}. Computation of evaluations of $l(\varphi(t_\text{prev},X^0))$ and $l(\varphi(t_\text{prev},X^0)))$  depends only on the $P$ part of representations of sets $\varphi(t_\text{prev},X^0)$ and
     $\varphi(t_\text{curr},X^0)$.
     We refer to \cite[Algorithm 1,3]{PoincareMapArticle}
     for further details about evaluation of $l.$
     Step (b) is not necessary but it significantly improves estimation of the algorithm result. For minimization of the crossing time we can use the rigorous version of the bisection method or the the rigorous Newton method \cite[Algorithm 5, Lemma 8]{PoincareMapArticle}.

     Steps (2) and (4) of above algorithm assure that assumptions (1) and (2) of Lemma \ref{lem:cross} are satisfied. In step (4) we can additionally return estimation of the
     set $[B]( X^E_P -[y^0])$ were $[B]$ is some interval matrix and $[y^0]$ is some interval vector. This allows to compute the resulting estimates in some system of coordinates defined on the section and ignore the component which is normal to the section.
     For more details of computation of  $[B]( X^E_P -y^0)$ see \cite[Algorithm 6]{PoincareMapArticle}.

%\subsection{Poincare map - fixed point}
%In this section we present theorem which
The following result allows us to deduce the existence of periodic orbit. Assumptions of this theorem can be checked with the use of previously described algorithm.
\begin{theorem}\label{th:PoincareFixedPoint}
Let $X^0$ be a nonempty, compact, convex subset of $Y$, such that for every $x^0\in X^0$ we have $l(x^0)=0$. Assume that for $t>0$ and $\tau>0$
the following holds:
\begin{enumerate}
    \item For every $x^0\in X^0$ we have
    $l(\varphi(t,x^0))<0$ and $l(\varphi(t+\tau,x^0))>0.$
    \item For every $x\in\varphi([t,t+\tau],X^0),$ we have
    $\sum_{i=1}^n\alpha_i F_i(x)>0.$
    \item For every $x\in \varphi([t,t+\tau],X^0)$ such that
    $l(x) = 0$ we have $x\in X^0.$
\end{enumerate}
Then there exist $x^* \in \varphi([t,t+\tau],X^0)$ and $T\in(t,t+\tau)$ such that
$l (x^*) = 0$ and $\varphi(T,x^*) = x^*$.
\end{theorem}
\begin{proof}
From Lemma \ref{lem:cross} we observe that for every $x\in X^0$ we have $l(\varphi(\tau_l(\varphi(t,x)),\varphi(t,x) )) = 0.$  Hence we can define the map $\mathcal{P}:X^0\to Y$ by the formula $\mathcal{P}(x^0) = \varphi(\tau_l(\varphi(t,x)),\varphi(t,x) ).$ The map $\mathcal{P}$ is continuous in the norm of the Banach space $Y$.
From assumption (3) we see that $\mathcal{P}(X^0) \subset X^0.$
As $X^0$ is compact and convex, the Schauder fixed-point theorem ensures the existence of a fixed point $x^*$. For this point we have $x^*\in \mathcal{P}(X^0)\subset \varphi([t,t+\tau],X^0)$ which concludes the proof.
\end{proof}
%With the following Remark additionally check if $T$ obtained in \ref{th:PoincareFixedPoint} is a fundamental period.
 \begin{remark}
    If in Theorem \ref{th:PoincareFixedPoint} we additionally assume that for some $t_1<t_2<t+\tau $:
    \begin{enumerate}
        \item For every $x\in\varphi([0,t_1],X_0)$ and $x\in \varphi([t_2,t+\tau],X_0)$ we have that $\sum_{i=1}^n\alpha_i F_i(x)>0$  .
        \item For every $t\in [t_1,t_2]$ we have $\varphi(t,X^0) \cap X^0 = \emptyset$.
    \end{enumerate}
    Then $T$ has to be a fundamental period for $x^*$.
\end{remark}

\subsection{Numerical approximation of periodic orbit}
The first step in constructing the initial data and the section which will be used in validating the assumptions of Theorem \ref{th:PoincareFixedPoint}, is finding the approximation of periodic using the Galerkin projection of \eqref{eq:AbstractBrusselator}.
This means that we need to find an initial data $x^*=(u^*,v^*)$ and a time $T^*$ such that the solution to the system
\begin{equation}\label{eq:ProjAbstractBrusselator}
    \begin{cases}
     \frac{d}{dt}P(u(t),v(t)) = LP(u(t),v(t)) + Pf( P(u(t),v(t))),\\
      (u(0),v(0)) = (u^*,v^*).
    \end{cases}
\end{equation}
is close to a periodic solution and $T^*$ is close to its period.
\begin{figure}[h]

\begin{subfigure}{0.45\textwidth}
    \includegraphics[width=1\linewidth, height=6cm]{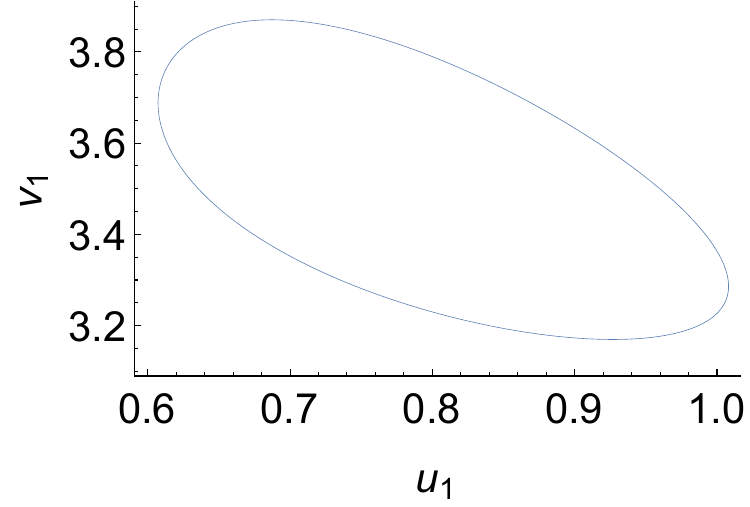}
    \caption{Projection of the numerical approximation of periodic orbit into two first modes $u_1,v_1,$
    for parameters $d_1,d_2,A,B$ from Theorem \ref{th:BrusselatorPeriodicOrbit}.
    }
\end{subfigure}
\quad
\begin{subfigure}{0.45\textwidth}
    \includegraphics[width=0.95\linewidth, height=6cm]{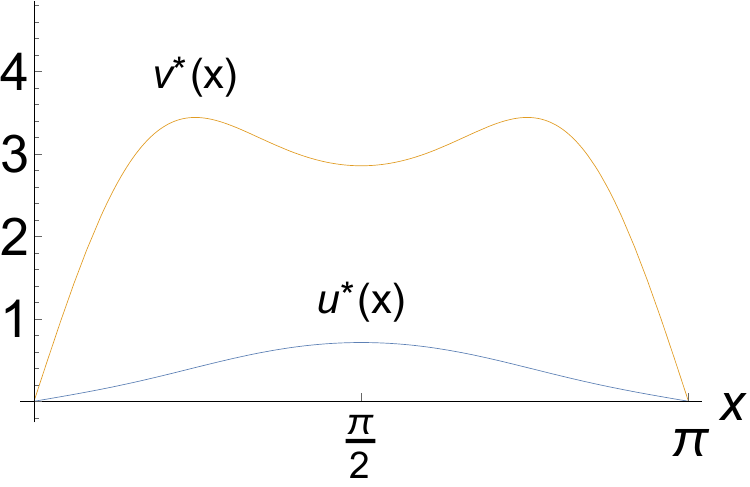}
\caption{Blue function corresponds to function $u^*(x)$ and orange to the function $v^*(x),$ around which we build the set of initial data.}
    \end{subfigure}

\label{fig:image2}
\caption{First two modes and initial data for the approximate periodic solution found by nonrigorous computations.}
\end{figure}

As, for the Brusselator system, we observe that periodic orbit is numerically attracting, it is enough to find approximation of attracting fixed point of some Poincare map.
We need to ensure that the section which defines this Poincare map intersects periodic orbit of \eqref{eq:ProjAbstractBrusselator}.
Additionally we are searching for $(u^*,v^*)$ for which only odd Fourier coefficients in the sine series are nonzero.
We chose to project the Brusselator system \eqref{eq:BrusselatorPDE} on the subspace of $L^2\times L^2$ spanned by the functions $\{(\sin(x),0),(\sin(3x),0),\ldots,(\sin(17x),0)\}$ and $\{(0,\sin(x)),(0,\sin(3x)),\ldots,(0,\sin(17x))\}.$

With this procedure we have found the following approximations of the fixed point of Poincar\'{e} map.
\begin{align*}
    u^*(x) &= 10^{-1}* 6.999\sin(x) - 10^{-2}* 8.170 \sin(3x)+
    -10^{-3}*5.377\sin(5x)+
    10^{-2}*1.325\sin(7x)+
    \\
    &+
    10^{-3}*1.050\sin(9x)
    -10^{-4}*2.585\sin(11 x)
     -10^{-6}*1.764\sin(13x) + 10^{-7}*5.029\sin(15x)
     \\
     &+
      10^{-8}*2.779\sin(17x),
\end{align*}
and
\begin{align*}
v^*(x) &= 3.869 \sin(x) +
1.136 \sin(3x) +
 10^{-1}* 1.017\sin(5x)
 - 10^{-3}*9.291 \sin(7x)
 \\&- 10^{-3}*1.297 \sin(9x)
 + 10^{-4}*1.960 \sin(11x) +
 10^{-5}*1.993 \sin(13x) -
 10^{-6}*4.109 \sin(15x)
 \\&- 10^{-7}*3.147 \sin(17x).
\end{align*}
The coefficients of $u^*$ and $v^*$ are written up to three decimal places.

\subsection{Defining a Poincar\'{e} map and constructing an initial set}\label{sec:ConstructingInitialSet}

We need to define a Poincar\'{e} map and an initial set that will be used to validate
Theorem \ref{th:BrusselatorPeriodicOrbit}.
The section is a mapping
$l:Y\to \mathbb{R}$ given by the formula
$l(x) = \sum_{i=1}^n \alpha_i(x_i-x^*_i),$ where the $x^*$ is the point from the numerical approximation of the periodic orbit.
The numbers $\alpha_i\in\mathbb{R}$ are chosen to assure that transversality condition (2) of Lemma \ref{lem:cross} holds. The method to choose the coefficients in the optimal way, such that the time of passing through the section is minimal is given in \cite[Theorem 18]{PoincareMapArticle}. We use this approach.
%optimal  according to the algorithm
%given by corresponding coefficients of numerically approximated left eigenvector of variation matrix which corresponds to eigenvalue equal $1.$
%This variational matrix corresponds to evolution of initial data $x^*$ by the time $x^*.$
%This choice minimizes the time for sets to cross the section (see \cite[Theorem 18]{PoincareMapArticle}).
To describe the initial set $X^0 = X_P^0 + X_Q^0 $ we define separately two sets $X_P^0$ and $X_Q^0.$
The set $X_P^0$ is defined as follows
\begin{equation*}
    X_P^0= x^* + Ar^0,
\end{equation*}
where $x^*$ is a (noninterval) vector, $A$ is a (noninterval) square matrix and $r^0$ is an interval vector. Note, that since both $r^0$ is an interval vector, $x^*$ is a vector, and $A$ is a matrix, $X_P^0$ is a bounded and convex set.

The first column $c_1$ of the   matrix $A$
is equal to $(\alpha_1,\ldots,\alpha_n).$ The subsequent columns $c_2,\ldots,c_n$ of this matrix 
constitute the coordinate system on section. They are assumed to satisfy the following two conditions
\begin{enumerate}
    \item For every $i\in\{2,\ldots,n\}$ we have $l(c_i+x^*) = 0,$
    \item
    For the linear functional $\hat{l}(x) = l(x+x^*)$  we have that $\text{span}\{c_2,\ldots,c_n\} = P(\ker \hat{l}).$
\end{enumerate}
The first coefficient of $r^0$ can be set to interval $[0,0]$ as our initial set should be on the previously defined section. The next coefficients describe size of the set in the section and can be fixed for example to intervals $[-\delta,\delta]$ where $\delta>0$. In the algorithm of the evolution of the set, we will take as the initial data the set $\overline{X}_P^0$ which is a superset of $X_P^0$ and is defined by $\overline{X}_P^0 = [x^*] + [A]r^0$, where $[A]$ is an interval matrix containing $A$ and $[x^*]$ is an interval vector containing $x^*$. The first column of $[A]$ is the interval vector containing $c_1$. The remaining interval columns denoted by $\{C_i\}_{i=2}^n$ are constructed in such a way that they are guaranteed to contain vectors which constitute the coordinate system $\{c_i\}_{i=2}^n$ on the section, i.e. $c_i\in C_i$ for $i\in \{2,\ldots,n \}$. 
In construction of columns $C_2,\ldots, C_n$ of the matrix $[A]$ we can use numerical approximation of the eigenvectors of the Poincar\'{e} map given by the section $l$.
 Additionally every numerical approximation of these eigenvectors is rigorously projected on the first column in order to assure that condition $(1)$ is satisfied for a certain $c_i\in C_i$. We also compute an interval matrix $[B]$ which is a rigorous interval inverse of the interval matrix $[A].$ Existence of this matrix ensures that condition (2) is satisfied. Matrix $[B]$ is also used in further steps.
 The set $X^0_Q$ is defined by infinite interval vectors with the polynomial decay of coefficients as presented in \eqref{eq:representationInfiniteVectors}.

\subsection{Computer assisted proof}
With the previously defined section $l$ and the initial data $X^0$ we check that the assumptions of Theorem \ref{th:PoincareFixedPoint} are satisfied. We use the algorithm described in Section \ref{sec:crossing} to estimate the image $\mathcal{P}(X^0) \subset X^1_P+ X^1_Q.$ This algorithm guarantees that assumptions (1) and (2) of Theorem \ref{th:PoincareFixedPoint} are satisfied. The set $X^1_P$ is returned in the form
\begin{equation*}
    X^1_P = [x^*] + [A]q^0.
\end{equation*}
 To obtain the interval vector $q^0$ we use the algorithm of computing a Poincar\'{e} map with the matrix $[B],$ which is an interval inverse of the matrix $[A],$ and $[y^0]$ equal to $[x^*]$.  To validate the assumption (3) of Theorem \ref{th:PoincareFixedPoint} it is enough to check that
 \begin{itemize}
     \item[(C1)]\label{as:C1} we have $q^0_i\subset r^0_i$  for $i\in\{2,\ldots,n\},$
     \item[(C2)]\label{as:C2} we have $X^1_Q\subset X^0_Q.$
 \end{itemize}
 In the computer assisted proof of Theorem \ref{th:BrusselatorPeriodicOrbit} we set
  \begin{equation*}
     H_{P}:= \text{span}
     \bigcup_{\stackrel{1\leq k\leq 59}{k\,\textrm{mod}\, 2 = 1}} \{ (\sin(kx),0),(0,\sin(kx))\},
 \end{equation*}
We define $H_Q$ as the orthogonal complement of $H_P,$
in the space
\begin{equation*}
    H:=\left\{(u,v)\in L^2\times L^2:\, u(x) = \sum_{k=1}^\infty u_{2k-1}\sin((2k-1)x),\, v(x) = \sum_{k=1}^\infty v_{2k-1}\sin((2k-1)x)\right\}.
\end{equation*}
Observe that
$\text{dim}(H_{P}) = 60.$
 The set $X_P^0$ is defined using the procedure described in Section \ref{sec:ConstructingInitialSet} applied to the Brusselator system. This is a bounded, closed and convex set in a finite dimensional space $H_P$.
To see that $q^0_i\subset r^0_i$ implies that $X^0_P\subset X^1_P$ note that $q^0\subset r^0$ implies that 
$$
[B](X_P^1-[x^*]) = q^0 \subset r^0 = A^{-1}(X^0_P-x^*).
$$
But, since $A^{-1}\in [B]$ and $x^*\in [x^*]$, the last inclusion implies that
$$
A^{-1}(X_P^1-x^*)\subset  A^{-1}(X^0_P-x^*),
$$
which guarantees the required inclusion $X_P^1\subset X_P^0$. 
 The part $X_Q^0$ is given as follows
 \begin{equation*}
 X_Q^0:=\left\{(u,v)\in H^Q:\;
     u_k\in \frac{[-1,1]}{|k|^5},
     \;
     v_k\in \frac{[-1,1]}{|k|^5}\quad\text{for}\; k>59
     \right\}.
 \end{equation*}
  The set $X^0 = X_P^0+X_Q^0$ must be compact in $Y=C_0\times C_0$. As it is closed in $Y$, it is sufficient to show that it is bounded in $H^1_0\times H^1_0$. To this end assume that $(u,v)\in X^0$. Then
\begin{align*}
(u_x,v_x) = \left(\sum_{\underset{k\; \textrm{mod}\; 2 = 1}{k=1}}^\infty k u_k \sin(kx),\sum_{\underset{k\; \textrm{mod}\; 2 = 1}{k=1}}^\infty k u_k \sin(kx)\right),
\end{align*}
and 
\begin{align*}
& \|(u,v)\|_{H^1_0\times H^1_0}^2 = \|(u_x,v_x)\|_{L^2\times L^2}^2 = \frac{\pi}{2}\sum_{\underset{k\; \textrm{mod}\; 2 = 1}{k=1}}^\infty k^2(|u_k|^2+|v_k|^2) \\
& \ \ \ \ \ \leq  \frac{59^2\pi}{2}\sum_{\underset{k\; \textrm{mod}\; 2 = 1}{k=1}}^{59} (|u_k|^2+|v_k|^2) + \frac{\pi}{2}\sum_{\underset{k\; \textrm{mod}\; 2 = 1}{k=61}}^\infty \frac{2}{k^8}.
\end{align*}
Since the last quantity is bounded uniformly with respect to the choice of $(u,v)\in X^0$, we get the required compactness of this set in $Y$.
 After the computation of Poincar\'{e} map we get the set $X_P^1 + X_Q^1$ with the following $Q$ part
\begin{equation*}
    X_Q^1:=\left\{(u,v)\in H^Q:\;
     u_i\in 10^{-14}\frac{[-3.46474, 3.46474]}{|i|^{6.5}},
     \;
     v_i\in 10^{-13}\frac{[-3.9024, 3.9024]}{|i|^{6.5}}\quad\text{for}\; i>59
     \right\}.
\end{equation*}
 The conditions \hyperref[as:C1]{(C1)},\;\hyperref[as:C2]{(C2)} are satisfied, so the main Theorem \ref{th:BrusselatorPeriodicOrbit} is validated. The fact that \hyperref[as:C1]{(C1)} holds is demonstrated in Tab. \ref{tab:1} where for brevity only the first $10$ coordinates are depicted. Due to the extra information on the periodic orbit obtained in the computer assisted proof, we can provide the following extended version of the main theorem.
 
\begin{theorem}
For parameters $d_1 = 0.2,\; d_2 = 0.02,\; A = 1,\;  B= 2$ the Brusselator system has a periodic solution $(\Bar{u}(t,x),\Bar{v}(t,x)),$ with the period $T\in [7.69666, 7.69667].$ The functions $\Bar{u}(t,x),\Bar{v}(t,x)$ are symmetric with respect to the point $x=\frac{\pi}{2}.$ Moreover the following estimates are true
\begin{align*}
    \sup_{t\in[0,T]}\norm{\Bar{u}(t)}_{L^2}&\leq 1.27261,\;
    \sup_{t\in[0,T]}\norm{\Bar{v}(t)}_{L^2}\leq 5.05587,
    \\
    \sup_{t\in[0,T]}\norm{\Bar{u}_x(t)}_{L^2}&\leq 1.35194,\;
    \sup_{t\in[0,T]}\norm{\Bar{v}_x(t)}_{L^2}\leq 6.75405,
    \\
    \sup_{t\in[0,T]}\norm{\Bar{u}(t)-{u}^*(t)}_{L^2}&\leq 0.00049664,\;
    \sup_{t\in[0,T]}\norm{\Bar{v}(t)-{v}^*(t)}_{L^2}\leq 0.000955005,
     \\
    \sup_{t\in[0,T]}\norm{\Bar{u}_x(t)-{u}_x^*(t)}_{L^2}&\leq 0.000546005,\;
    \sup_{t\in[0,T]}\norm{\Bar{v}_x(t)-{v}_x^*(t)}_{L^2}\leq 0.00141263,
\end{align*}
where $(u^*(t,x),v^*(t,x))$ is the solution to the Brusselator system with the initial data $u^*(0,x)=u^*(x),\; v^*(0,x)=v^*(x),$ and the same parameters as above.

%u%. Moreover this periodic orbit is symmetric with respect to line $y = \frac{\pi}{2}$.
\end{theorem}
\begin{center}

\begin{table}[h]
        \centering
        
\begin{tabular}{|p{0.07\textwidth}|p{0.24\textwidth}|p{0.4\textwidth}|}
\hline 
  & $\displaystyle r_{0}$ & $\displaystyle q_{0}$ \\
\hline 
 1. & $\displaystyle [ 0,0]$ & $\displaystyle 10^{-6}[-1.10436, 1.21675]$ \\
\hline 
 2. & $\displaystyle 10^{-5}[ -1\ ,\ 1]$ & $\displaystyle 10^{-6}[-6.87377, 6.87371]$ \\
\hline 
 3. & $\displaystyle 10^{-5}[ -1\ ,\ 1]$ & $\displaystyle 10^{-6}[-1.1776, 1.17766]$ \\
\hline 
 4. & $\displaystyle 10^{-5}[ -1\ ,\ 1]$ & $\displaystyle 10^{-11}[-2.8735, -2.36735]$ \\
\hline 
 5. & $\displaystyle 10^{-5}[ -1\ ,\ 1]$ & $\displaystyle 10^{-9}[-5.40597, 5.43755]$ \\
\hline 
 6. & $\displaystyle 10^{-5}[ -1\ ,\ 1]$ & $\displaystyle 10^{-9}[-1.60821, 1.59499]$ \\
\hline 
 7. & $\displaystyle 10^{-5}[ -1\ ,\ 1]$ & $\displaystyle 10^{-10}[-1.56502, 1.59018]$ \\
\hline 
 8. & $\displaystyle 10^{-5}[ -1\ ,\ 1]$ & $\displaystyle 10^{-11}[-4.72352, 6.65989]$ \\
\hline 
 9. & $\displaystyle 10^{-5}[ -1\ ,\ 1]$ & $\displaystyle 10^{-11}[-5.29417, -1.37926]$ \\
\hline 
 10. & $\displaystyle 10^{-5}[ -1\ ,\ 1]$ & $\displaystyle 10^{-11}[0.131355, 2.59156]$ \\
 \hline
\end{tabular}
       \caption{First 10  coordinates of interval vectors $r_0$ and computed $q_0$ which correspond to the sets $X^0_P$ and $X^1_P$ in the computer assisted proof of Theorem \ref{th:BrusselatorPeriodicOrbit} for parameters $d_1=0.2,d_2=0.02,A=1,B=2.$} \label{tab:1}
        \end{table}
\end{center}

\section{Numerical and rigorous results for other parameter values}\label{sec:numerical}
In this section we discuss some numerical observations and rigorous results concerning the Brusselator system with various parameter values. We  conduct the computer assisted proof of existence of periodic orbit for the parameter values from the set $\mathcal{A}= \mathcal{A}_1\cup\mathcal{A}_2\cup\mathcal{A}_3,$ of parameters $(d_1,d_2,A,B)$, where
\begin{align*}
     &\mathcal{A}_1 = \left\{\left(0.2,0.02,1,2+\frac{i}{10}\right): i\in \{0,\ldots,11\} \right\},
     \\
     &\mathcal{A}_2 = \left\{\left(1,\frac{1}{64},1,2.71\right),\left(1,\frac{1}{64},1,2.83\right),\left(1,\frac{1}{64},1,2.84\right)\right\},
     \quad
\mathcal{A}_3 =  { \{ (0.02,0.02,1,2) \} }.
\end{align*}

We have rigorously proved the existence of the  periodic orbits for all parameters from the set $\mathcal{A}$. Values in the set $\mathcal{A}_1$ correspond to the slow-fast behavior of the system, ones in the set $\mathcal{A}_2$ correspond to the period-doubling  bifurcation and cross-validation with the results of Arioli \cite{ArioliBrusselator}, and  the paramters in $\mathcal{A}_3$ are related with the numerical experiments where we find attracting torus. 

In the grid set $\mathcal{A}_1$ we fix the parameters $d_1, d_2, A$ and we increase the parameter $B$ by $0.1$ starting from its value $B=2$, which corresponds to the periodic orbit found in our main Theorem \ref{th:BrusselatorPeriodicOrbit}, and ending at the value $B=3.1$. For the corresponding planar ODE \eqref{eq:BrusselatorODE}, it is known \cite{BrusselatorSlowFast} that the slow-fast dynamics of the system increases  when the parameter $B$ grows. We observe the same intensification of slow-fast behavior upon the increase of parameter $B$ also for the Brusselator with diffusion. This is depicted in Fig. \ref{fig:OrbitsPlots1} and \ref{fig:OrbitsPlots2}.
The slow-fast behavior is especially visible in the higher Fourier modes, cf. Fig. \ref{fig:fourModesU}. We also stress that as the value of $B$ increases, we  require more modes to accurately represent the solution. Therefore, to successfully carry out the computer-assisted proof, we must increase the dimension of the inclusion as the value of $B$ increases, and, consequently the computation time lengthens. This is shown in Tab.\ref{tab:incltim}.

\begin{center}
	
	\begin{table}[!h]
		\centering
		
		\begin{tabular}{|p{0.22\textwidth}|p{0.22\textwidth}|p{0.22\textwidth}|}
			\hline
			Value of parameter $\displaystyle B$ & Number of variables in the inclusion \eqref{inclusion} & Computation time in seconds \\
			\hline
			$\displaystyle 2.1$ & $\displaystyle 14$ & $\displaystyle 298$ $\displaystyle s$ \\
			\hline
			$\displaystyle 2.3$ & $\displaystyle 22$ & $\displaystyle 332$ $\displaystyle s$ \\
			\hline
			$\displaystyle 2.5$ & $\displaystyle 22$ & $\displaystyle 341$ $\displaystyle s$ \\
			\hline
			$\displaystyle 2.7$ & $\displaystyle 30$ & $\displaystyle 477\ s$ \\
			\hline
			$\displaystyle 2.9$ & $\displaystyle 32$ & $\displaystyle 534\ s$ \\
			\hline
			$\displaystyle 3.1$ & $\displaystyle 40$ & $\displaystyle 771\ s$ \\
			\hline
		\end{tabular}
		\caption{Selected values of parameter $B$ used in the computer assisted proof of Theorem \ref{thm:morepars}, the dimensions of corresponding inclusion \eqref{inclusion} and the computation times needed to rigorously validate the periodic orbit existence. The computations were single-threaded and realized on the processor
			Intel(R) Core(TM) i5-4200M CPU @ 2.50GHz.}\label{tab:incltim}
\end{table}
\end{center}

The result on the existence of the priodic orbits for parameters in $\mathcal{A}_1$ is contained in the next theorem.
\begin{theorem}\label{thm:morepars}
	The Brusselator system has a periodic orbit for  $(d_1,d_2,A,B)\in\mathcal{A}_1 $.
\end{theorem}

\color{black}

\begin{figure}[H]
    \centering
    \includegraphics{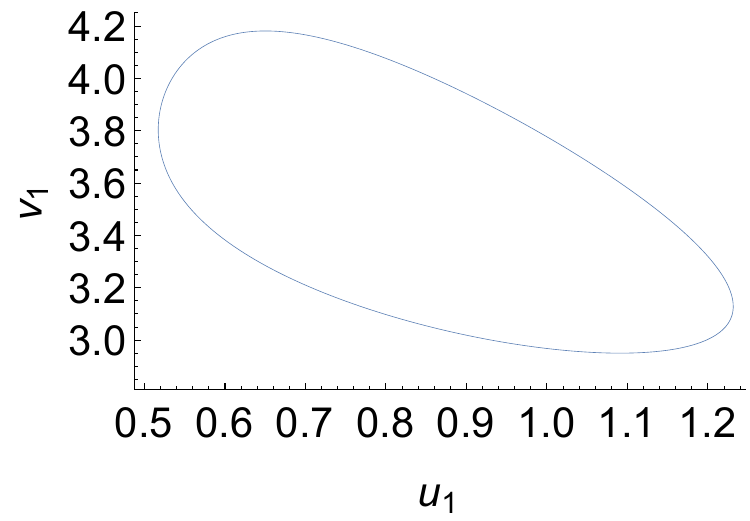}
    \quad
    \includegraphics{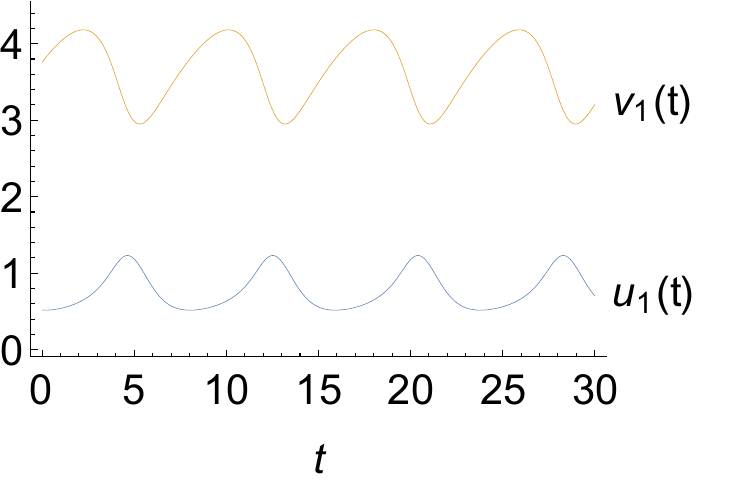}
    \includegraphics{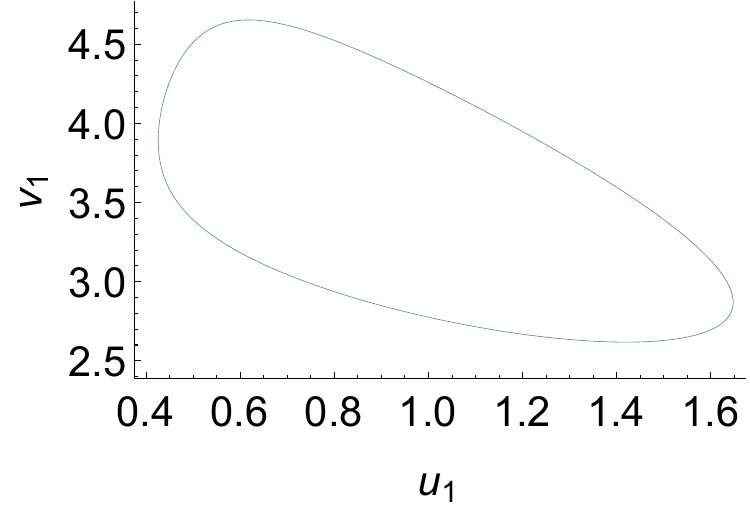}
    \quad
    \includegraphics{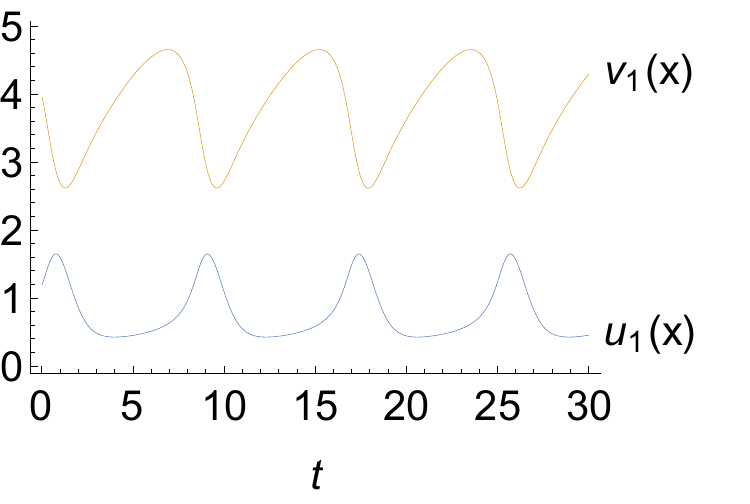}

    \caption{Plots of projection of numerical approximation of the periodic orbits into first modes of $u$ and $v$ (left) and the evolution in time of the first Fourier modes (right) for
    $d_1 = 0.2$, $d_2=0.02$, $B=2.1$ (top) and $B=2.3$ (bottom), $A=1.$ }
    \label{fig:OrbitsPlots1}
\end{figure}

\begin{figure}[H]
    \centering
        \includegraphics{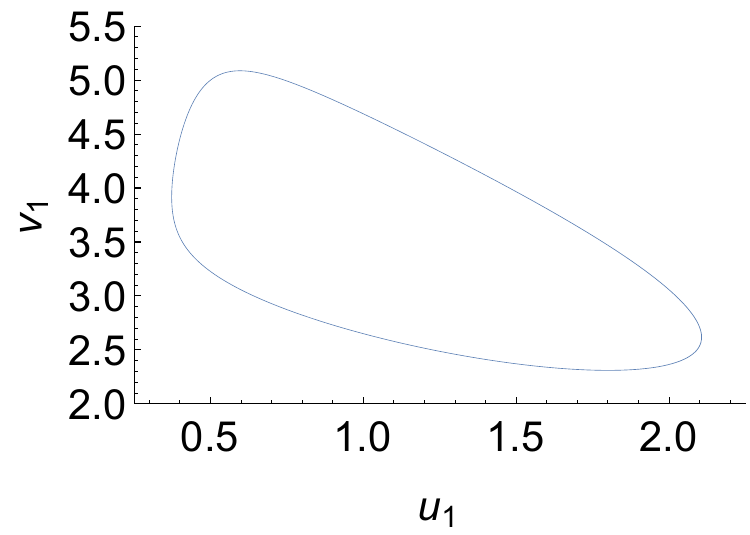}
    \quad
    \includegraphics{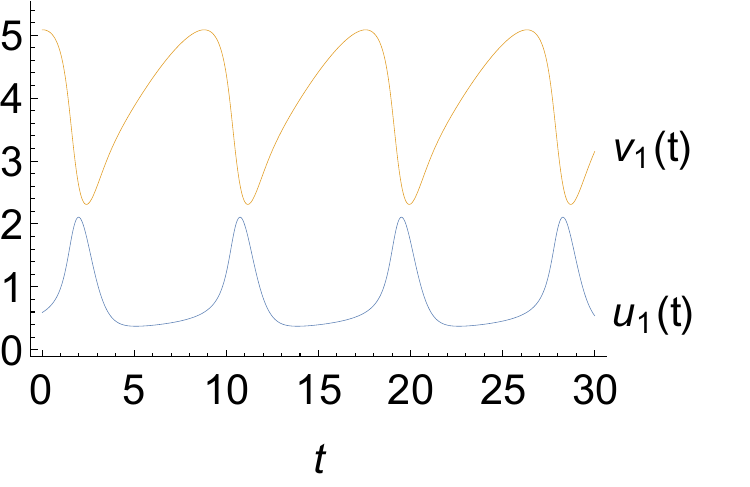}
    \noindent
    \includegraphics{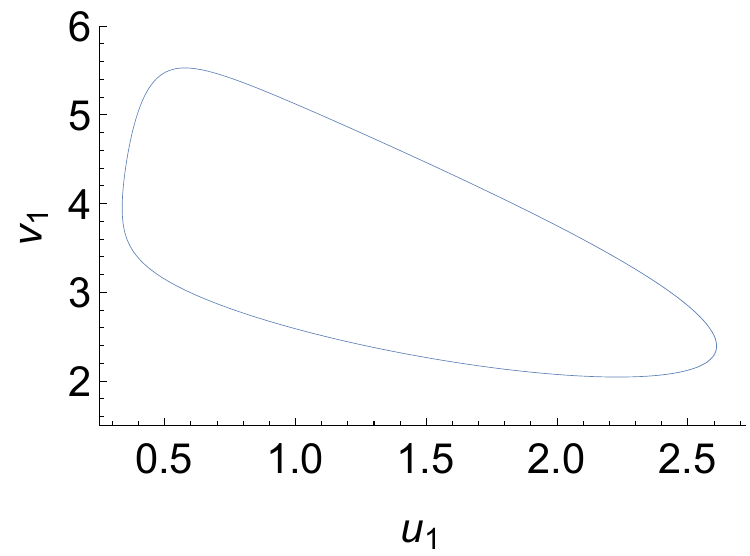}
    \quad
    \includegraphics{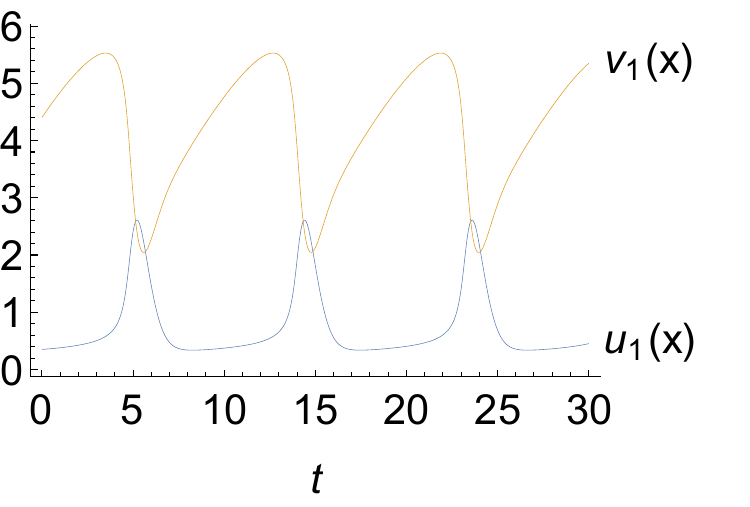}
    \includegraphics{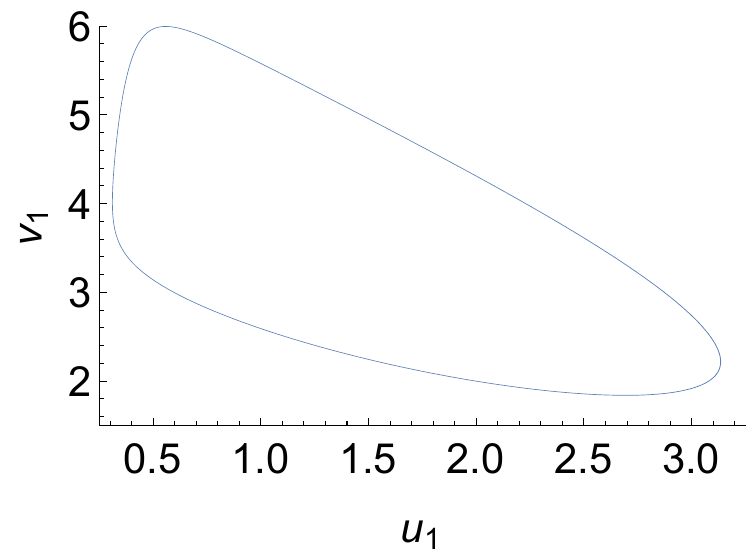}
    \quad
    \includegraphics{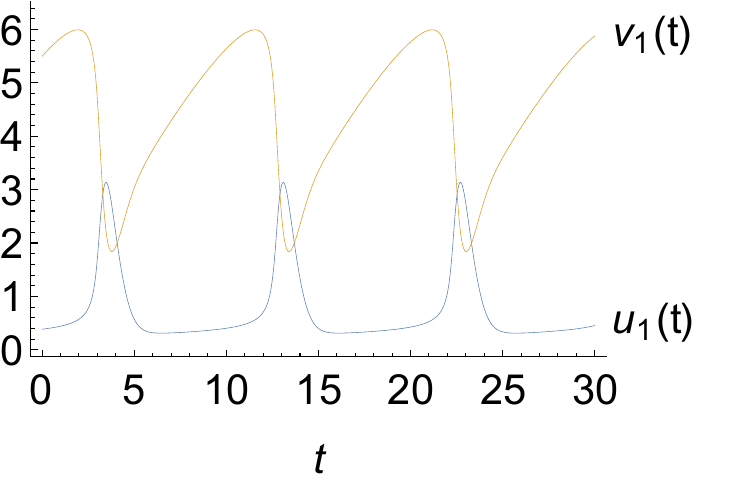}
    \quad
    \includegraphics{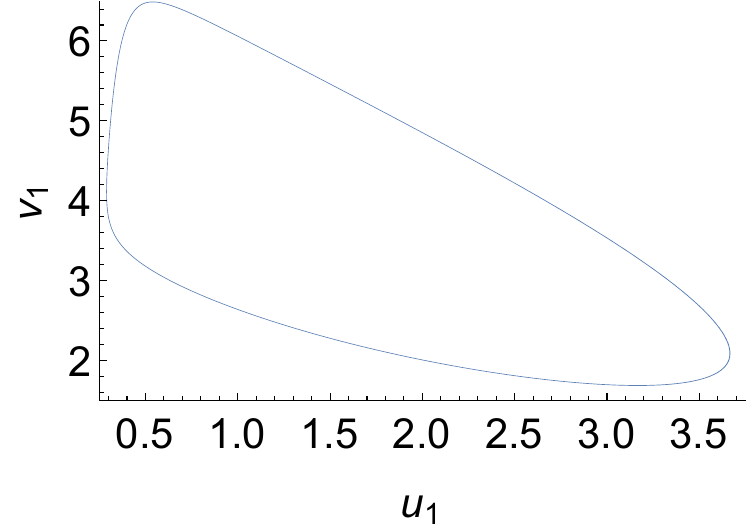}
    \quad
    \includegraphics{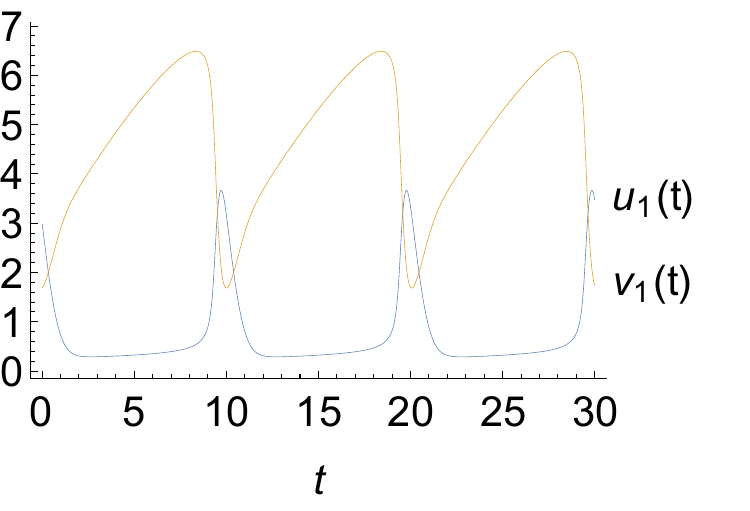}
    \caption{Plots of projection of numerical approximation of the periodic orbits into first modes of $u$ and $v$ (left) and the evolution in time of the first Fourier modes (right) for
    $d_1 = 0.2,\;d_2=0.02,\;B=2.5,\;2.7,\;2.9,\;3.1$ (from top to bottom), $A=1.$}
    \label{fig:OrbitsPlots2}
\end{figure}
\begin{figure}[H]
    \centering
    \includegraphics{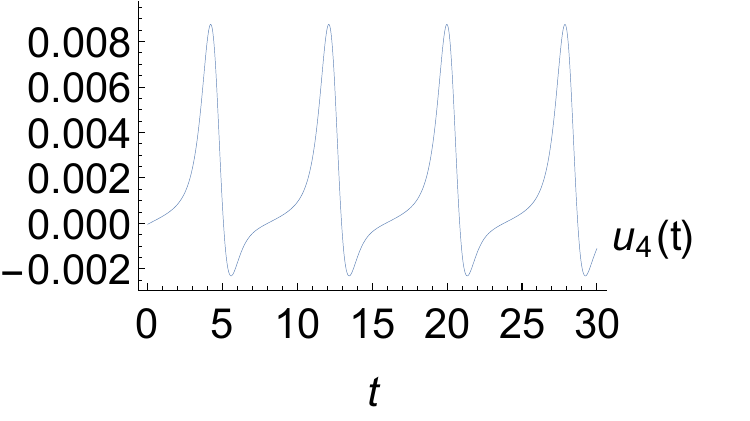}
    \quad
    \includegraphics{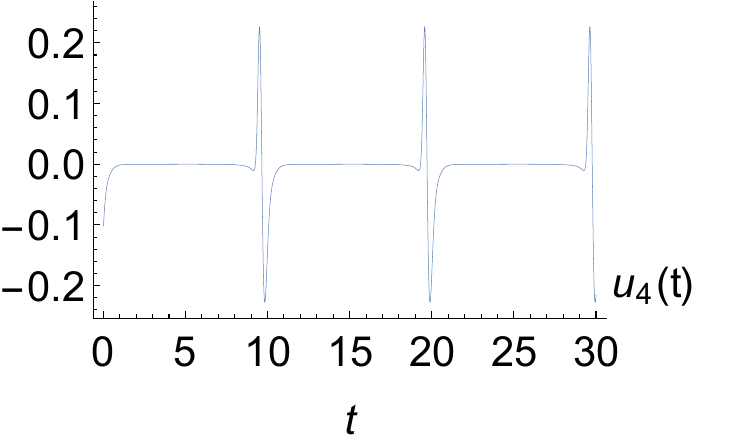}
    \caption{The evolution of fourth mode of $u$ for $B = 2.1$ (left picture) and $B=3.1$ (right picture).}
    \label{fig:fourModesU}
\end{figure}

In \cite[Theorem 2]{ArioliBrusselator} Arioli used computer assisted method and proved the existence of periodic orbit  for parameters $d_1=1,\,d_2 = \frac{1}{64},\,B\in[2.6993,2.7419],\,A=1.$ We have checked his result for $B=2.71$, and we found that both results correspond to each other in terms of the found period of the orbit. Our resuls, which succesfully cross-validates both approaches, is contained in the following theorem.

\begin{theorem}\label{th:Arioli}
For parameters $d_1 = 1,\; d_2 = \frac{1}{64},\; A = 1,\;  B= 2.71$ the Brusselator system has a periodic solution $(\Bar{u}(t,x),\Bar{v}(t,x)),$ with the period $T\in [10.4549, 10.455].$
The functions $\Bar{u}(t,x),\Bar{v}(t,x)$ are symmetric with respect to the point $x=\frac{\pi}{2}.$ Moreover the following estimates are true
\begin{align*}
    \sup_{t\in[0,T]}\norm{\Bar{u}(t)}_{L^2}&\leq 0.600569,\;
    \sup_{t\in[0,T]}\norm{\Bar{v}(t)}_{L^2}\leq 5.05587,
    \\
    \sup_{t\in[0,T]}\norm{\Bar{u}_x(t)}_{L^2}&\leq 10.0529,\;
    \sup_{t\in[0,T]}\norm{\Bar{v}_x(t)}_{L^2}\leq 11.5792,
    \\
    \sup_{t\in[0,T]}\norm{\Bar{u}(t)-{u}^*(t)}_{L^2}&\leq 10^{-5}*8.69037,\;
    \sup_{t\in[0,T]}\norm{\Bar{v}(t)-{v}^*(t)}_{L^2}\leq 0.000295375,
     \\
    \sup_{t\in[0,T]}\norm{\Bar{u}_x(t)-{u}_x^*(t)}_{L^2}&\leq 10^{-5}*8.99234,\;
    \sup_{t\in[0,T]}\norm{\Bar{v}_x(t)-{v}_x^*(t)}_{L^2}\leq 0.000450742,
\end{align*}
where $(u^*(t,x),v^*(t,x))$ is the solution to the Brusselator system with the initial data

\begin{align*}
    u^*(0,x) &= 0.43 \sin (x)-0.0231361 \sin (3 x)-0.00129933 \sin (5 x)+10^{-5} *7.48643 \sin (7 x)
    \\
    &+10^{-6}*7.466799 \sin (9 x)-10^{-7}*2.998759 \sin (11 x)-10^{-8}* 3.89921 \sin (13 x)
    \\
    &+10^{-9}*1.15918 \sin (15x)+10^{-10}*1.98398 \sin (17 x),
\end{align*}
\begin{align*}
    v^*(0,x) &= 7.85996 \sin (x)+1.59666 \sin (3 x)+0.091348 \sin (5 x)-0.0041776 \sin (7 x)
    \\
    &-10^{-4}*4.73146 \sin (9 x)+10^{-5}*1.66984 \sin (11 x)+10^{-6}*2.42261\sin (13 x)
    \\
    &-10^{-8}* 6.5405 \sin (15 x)-10^{-8}*1.23104 \sin (17 x),
\end{align*}
 and the same parameters as above.

\end{theorem}

\begin{figure}[H]

\begin{subfigure}{0.4\textwidth}
    \includegraphics[width=0.9\linewidth, height=6cm]{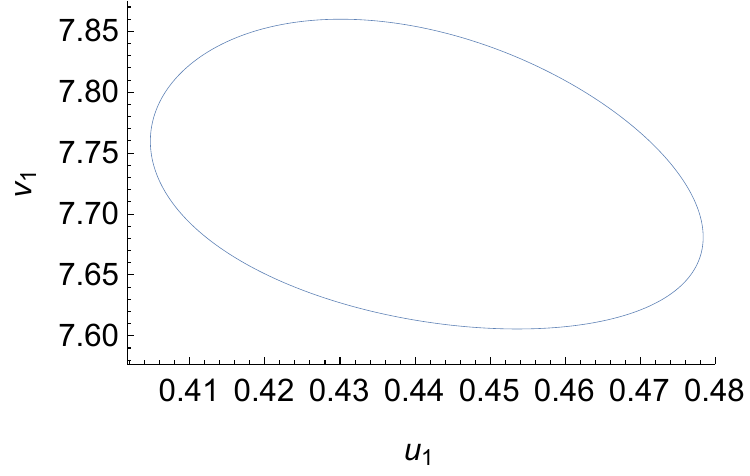}
    \caption{Projection of the numerical approximation of periodic orbit on two first modes $u_1,v_1,$
    for parameters  from Theorem \ref{th:Arioli}.
    }
\end{subfigure}
\quad
\begin{subfigure}{0.4\textwidth}
    \includegraphics[width=0.9\linewidth, height=6cm]{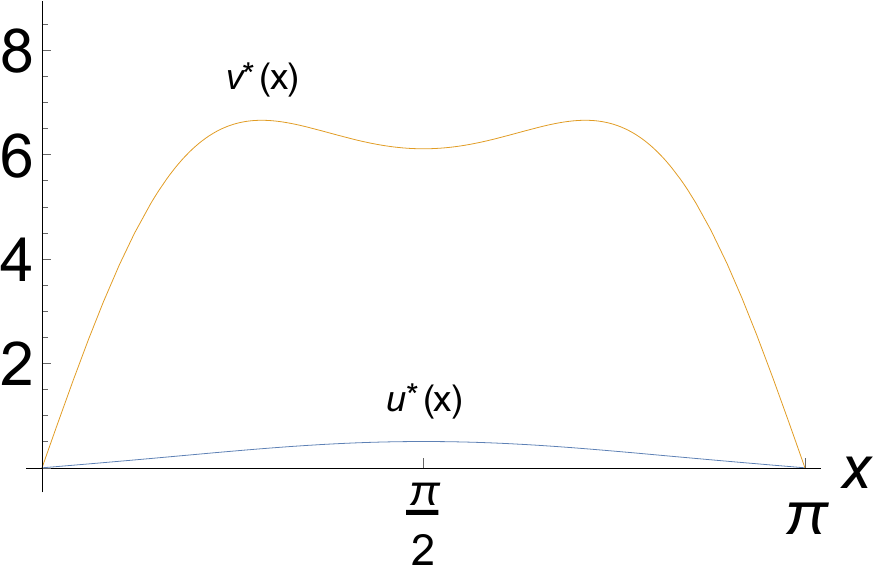}
\caption{Blue plot corresponds to  $u^*(0,x)$ and orange to the  $v^*(0,x)$ in Theorem \ref{th:Arioli}.  }
    \end{subfigure}
\caption{Periodic orbit from Theorem \ref{th:Arioli}, i.e. for $B=2.71$.}

\label{fig:image4}
\end{figure}
\begin{figure}[H]
    \centering
    \includegraphics{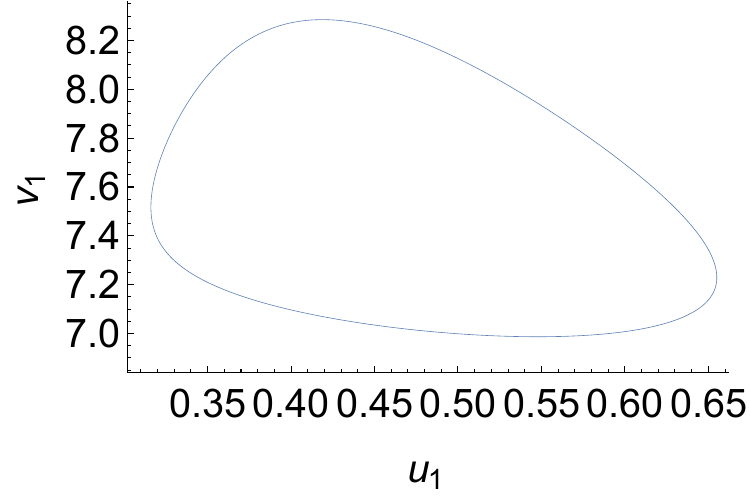}
    \quad
    \includegraphics{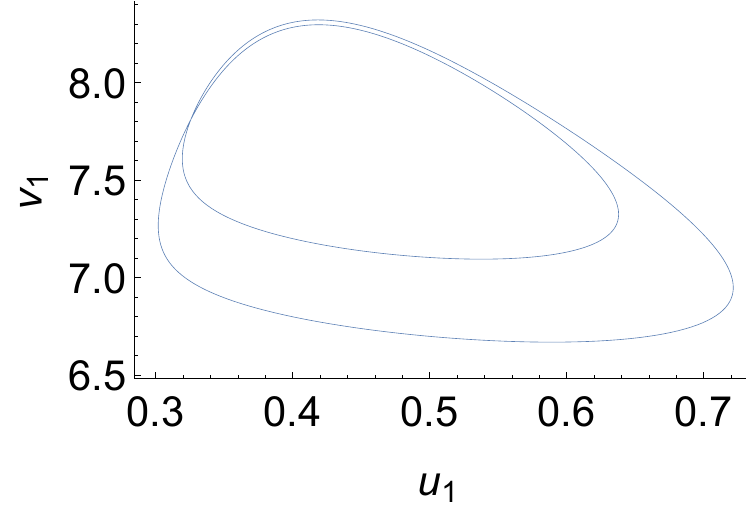}
    \caption{Numerically attracting periodic orbit for $B=2.83$ (left) and $B = 2.84$ (right).}
    \label{fig:my_label}
\end{figure}

Arioli observed a period doubling bifurcation, a phenomenon that cannot occur in the  planar ODE \eqref{eq:BrusselatorODE}. Thus, the dynamics of \eqref{eq:BrusselatorPDE} is expected to be more complicated than that of the planar ODE \eqref{eq:BrusselatorODE}. Although we do not rigorously prove the bifurcation, we show that the minimal period of the found orbits approximately doubles with a small increase in the parameter $B$, as seen in Theorem \ref{th:podwojenieOkresu}.
Specifically, as we increase $B$ and keep other parameters fixed $d_1 = 1,\; d_2 = \frac{1}{64},\; A = 1$, from numerical simulations we observe that the system goes through the period doubling bifurcation.
The bifurcation appears to occur between $B=2.83$ and $B=2.84$. We have proven the following theorem about periodic orbits for these parameters.
\begin{theorem}\label{th:podwojenieOkresu}
For $d_1=1,\;d_2=\frac{1}{64}\;,$ $B = 2.83\;A= 1$ there exist a periodic orbit with its fundamental period in the interval
$[13.2128, 13.2130]$. For $d_1=1,\;d_2=\frac{1}{64}\;,$ $B = 2.84,\;A= 1$ there exist a periodic orbit with its fundamental period in the interval $[27.2436, 27.2439]$.
\end{theorem}
 Note that we only obtain rigorous computer assisted proofs for the parameter values $B$ before and after the expected bifurcation. We believe that conducting computer assisted proof of the bifurcation existence would be interesting and challenging problem.

 Other nontrivial dynamics of the Brusselator PDE \eqref{eq:BrusselatorPDE} was investigated in \cite{BrusselatorChaosAnd2DTori}, where the numerical evidence on the existence of 2-dimensional attracting tori was shown. We observe, only numerically, the same phenomenon, which exists only in PDE Brusselator model and not in the corresponding planar ODE \eqref{eq:BrusselatorODE}.
 First, we rigorously observe the existence of the periodic orbit when the diffusion rates $d_1$ and $d_2$ are equal to each other. We have proved that for parameter values $d_1=0.02, d_2=0.02, B = 2, A= 1$.
If we decrease the diffusion rates $d_1$ and $d_2$ (keeping them equal to each other)
we numerically observe the emergence of two dimensional attracting torus, cf. Fig. \ref{fig:2dTori}. This is a numerical confirmation of the phenomenon which was firstly observed in  \cite{BrusselatorChaosAnd2DTori}. Note that the same paper also contains the numerical evidence for the presence of chaos in the same system of equations. These dynamical phenomena are further interesting challenge for computer
assisted proofs for the Brusselator system.

\begin{figure}[H]
    \centering
    \includegraphics[width=0.5\textwidth]{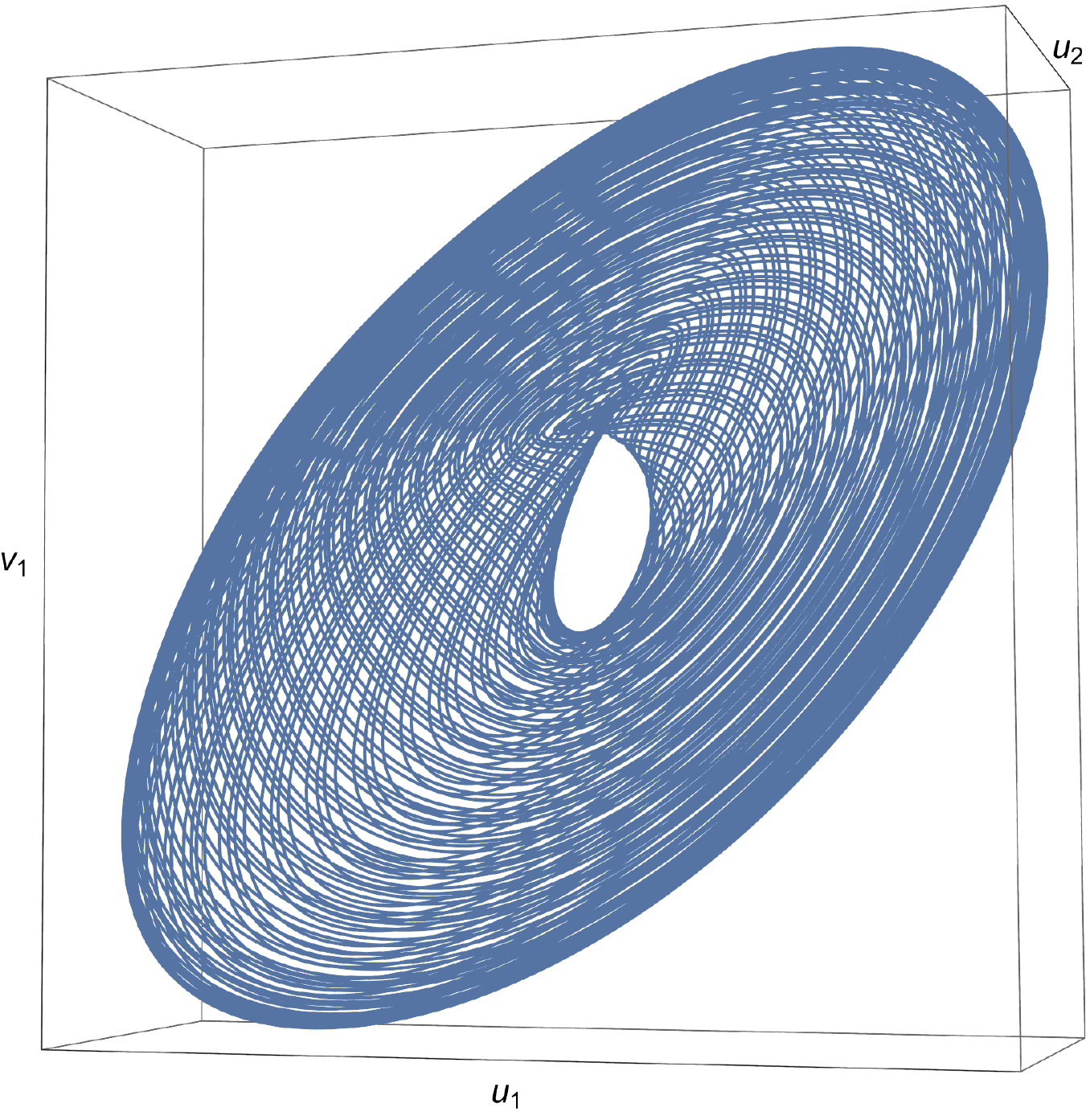}
    \caption{Two dimensional numerically attracting torus observed in for $d_1= 0.009,\; d_2=0.009,\; B=2.1,\;A=1.$}
    \label{fig:2dTori}
\end{figure}

\section{Algebra}\label{sec:algebra}
This section contains the technical results on the estimates of the convolutions of the sine and cosine Fourier series represented with uniform estimates on the decay of coefficients. Such operations are performed during the calculations in the algorithm as part of the computer-assisted proofs, where we require uniform estimates on the tails of the series.  We work with sequences $\{u_i\}_{i=1}^\infty$ such that first $n$ coefficients are given by some numbers or intervals and the remainders of the sequence satisfy $u_i\in\frac{[C_u^-,C_u^+]}{i^s}$ for some $C_u^-\leq C_u^+$ and $s\in\mathbb{R}.$

The decay of the Fourier coefficients for smooth periodic functions is related with their regularity: if a periodic function  $u:\mathbb{R}\to \mathbb{R}$ is of class $C^s$, then its coefficients must decay as $\frac{1}{i^s}.$ Clearly, the product of two $C^s$ functions also has regularity $C^s$. This is related with our results of this section, which state that if two functions, represented in the sine or cosine Fourier series have some given decay of the Fourier coefficients of the form $O\left(\frac{1}{i^s}\right)$ for $s>1$ then their product must have the same decay. Moreover we provide the exact estimates for the Fourier coefficients of the product: such estimates  are needed in rigorous computations of nonlinear polynomial terms present in the equations.

The following results are basic, so we skip the proofs. We  use them several times is the following considerations.
\begin{proposition}
Let $s>1.$ The following inequality holds
\begin{equation}\label{eq:sumEstimation}
\sum_{i=1+n}^\infty \frac{1}{i^s}\leq \frac{n^{1-s}}{s-1} .
\end{equation}
\end{proposition}
\begin{proposition}
Let $s>1$. If $a,b>0$ then the following inequality holds
\begin{equation}\label{eq:inequality}
    (a+b)^s\leq 2^{s-1}(a^s+b^s).
\end{equation}
\end{proposition}
The following lemma gives us the estimates on the result of multiplication of $u$ and $v$ which are both represented in the sine Fourier series. The first $n$ coefficients of $u$ and $v$ are given explicitly and the coefficients indexed by numbers larger than $n$ are expressed by the polynomial decay.
\begin{lemma}\label{lem:sinTimesSin}
Assume that
\begin{equation}
u(x) = \sum_{i=1}^\infty u_i\sin(ix), \quad v(x) = \sum_{i=1}^\infty v_i \sin(ix).
\end{equation}
Moreover we assume that for some
$n\in\mathbb{N}$ and $s>1$ the following estimates hold
\begin{equation}
    u_i \in \frac{C_u[-1,1]}{i^s}\quad \textrm{and}\quad v_i \in \frac{C_v[-1,1]}{i^s}\quad\text{for $i>n,$}
\end{equation}
where $C_v>0,C_u>0$. Then
\begin{equation}
(uv)(x) = (uv)_0 + \sum_{k=1}^\infty (uv)_k \cos(kx),
\end{equation}
with
\begin{equation}
    (uv)_0 = \frac{1}{2}\sum_{i=1}^\infty u_iv_i\quad
    (uv)_k =
    \frac{1}{2}\sum_{i=1}^\infty u_{i+k}v_i +
    \frac{1}{2}\sum_{i=1}^\infty u_{i}v_{i+k}-
    \frac{1}{2}\sum_{i=1}^{k-1} u_{i}v_{k-i}.
\end{equation}
and the following estimates hold for the coefficients of the product. For $k=0$ we have
    \begin{equation}
        (uv)_0 \in \frac{1}{2}\sum_{i=1}^n u_i v_i + \frac{C_uC_v}{2}\frac{ n^{1-2s}}{ 2s - 1 }[-1,1],
    \end{equation}
for $ 1\leq k \leq 2n$ we have
    \begin{equation}
        (uv)_k \in
        \frac{1}{2}\sum_{i=1}^n u_i v_{i+k}  +
        \frac{1}{2}\sum_{i=1}^n u_{i+k} v_{i} - \frac{1}{2}\sum_{i=1}^{k-1}u_{i}v_{k-i} + C_uC_v\frac{ n^{1-2s}}{ 2s - 1 } [-1,1],
    \end{equation}
    and for $k>2n$ we have
    \begin{equation}
        (uv)_k \in \frac{D[-1,1]}{k^s},\quad D =
    \frac{1}{2}\left(
    \sum_{i=1}^n (C_u|v_i| + C_v|u_i|)
    \left(1+\left(\frac{2n+1}{2n+1-i}\right)^s\right)+
    C_uC_v(2+2^s)\frac{n^{1-s}}{s-1}
    \right).
    \end{equation}
\end{lemma}
The above lemma provides explicit formulas for estimating the coefficients with indexes $0$ to $2n$ of the cosine Fourier series.
This is motivated by the fact that if the sine expansions of $u$ and $w$ are finite and concentrate on the first $n$ coefficients, then the representation of $uv$ is also finite and only first $2n+1$ coefficients are nonzero. In the formulas in the above lemma, whenever the terms $u_i$ and $v_i$ appear with $i > n$ in the computation, we substitute them with the intervals $\frac{C_u}{i^s}[-1,1]$ and $\frac{C_v}{i^s}[-1,1]$, respectively. Additionally, we can use other available estimates of $u_i$. For example we can use the fact that $u_i\in \frac{[C_u^-,C_u^+]}{i^s}$ (with upper and lower bound different from each other), or that $u_i$ is zero for odd coefficients.

The  coefficients for $k > 2n$ of the cosine expansion in the above lemma are given by the uniform polynomial decay with the same rate as the sine series  for $u$ and $v.$

\begin{proof}[proof of Lemma \ref{lem:cosTimesSin}]
We have
\begin{align*}
    uv &= \left(\sum _{i_1=1}^\infty \sin(i_1x)\right)
         \left(\sum _{i_2=1}^\infty \sin(i_2x)\right) =
         \frac{1}{2}
         \sum_{i_1,i_2=1}^{\infty}
         u_{i_1}v_{i_2}
         \cos((i_1-i_2) x) -
         \frac{1}{2}
         \sum_{i_1,i_2=1}^{\infty}
         u_{i_1}v_{i_2}
         \cos((i_1+i_2) x)
        \\&=
        \frac{1}{2}\sum_{i_1 = i_2}u_{i_1} v_{i_2}
        +
        \frac{1}{2}\sum_{k=1}^\infty \left(
        \sum_{i_1 - i_2 = k}u_{i_1} v_{i_2}
        +
        \sum_{i_1 - i_2 = -k}u_{i_1} v_{i_2}
        -
        \sum_{i_1 - i_2 = k}u_{i_1} v_{i_2}
        \right) \cos(kx).
\end{align*}
We express all coefficients of the resultant cosine series separately, using the formulas
\begin{equation*}
    (uv)_0 = \frac{1}{2}\sum_{i=1}^\infty u_iv_i,
\end{equation*}
and
\begin{equation*}
    (uv)_k =
    \frac{1}{2}\sum_{i=1}^\infty u_{i+k}v_i +
    \frac{1}{2}\sum_{i=1}^\infty u_{i}v_{i+k}-
    \frac{1}{2}\sum_{i=1}^{k-1} u_{i}v_{k-i},
\end{equation*}
for $k\geq 1.$
Observe that for $i>n$ there holds
\begin{equation*}
    |u_iv_i|\leq \frac{C_vC_u}{i^{2s}}.
\end{equation*}
So,  we obtain
\begin{equation*}
      (uv)_0 = \sum_{i=1}^n u_iv_i +\sum_{i=n+1}^\infty u_iv_i\in
      \sum_{i=1}^n u_iv_i +C_uC_v\sum_{i=n+1}\frac{1}{i^{2s}} [-1,1]
     \subset \sum_{i=1}^n u_iv_i +
     C_uC_v\frac{n^{1-2s}}{-1+2s}[-1,1].
\end{equation*}
Now, observe that for $i>n$ the following estimates hold
\begin{equation*}
    |u_iv_{i+k}|\leq \frac{C_vC_u}{i^{2s}},
    \quad
    |u_{i+k}v_i|\leq \frac{C_vC_u}{i^{2s}},
\end{equation*}
whereas we deduce that for $1 \leq k\leq 2n$ we have
\begin{align*}
(uv)_k &=
    \frac{1}{2}\sum_{i=1}^n u_{i+k}v_i +
    \frac{1}{2}\sum_{i=1}^n u_{i}v_{k+i}
    -\frac{1}{2}\sum_{i=1}^{k-1} u_{i}v_{k-i}
    +
    \frac{1}{2}\sum_{i=n+1}^\infty (u_{i+k}v_i + u_{i}v_{i+k})\\
     &\in
     \frac{1}{2}\sum_{i=1}^n u_{i+k}v_i +
    \frac{1}{2}\sum_{i=1}^n u_{i}v_{k+i}
    -\frac{1}{2}\sum_{i=1}^{k-1} u_{i}v_{k-i}
    + C_vC_u\frac{n^{1-2s}}{2s-1}[-1,1].
\end{align*}
Finally, for $k > 2n$ we need to estimate the sums in the expression
\begin{align*}
    (uv)_k=
    \frac{1}{2}
     &\Bigg[
    \sum_{i=1}^n u_{i+k}v_i +
    \sum_{i=1}^n u_{i}v_{i+k}+
    \sum_{i=n+1}^\infty u_{i+k}v_i+
    \sum_{i=n+1}^\infty u_{i}v_{i+k}\\
    &+
    \sum_{i=1}^n u_{k-i}v_i +
    \sum_{i=k-n}^{k-1} u_{k-i}v_i +
    \sum_{i=n+1}^{k-n-1} u_{k-i}v_i
    \Bigg].
\end{align*}
To this end observe, that we have
\begin{equation*}
    \sum_{i=1}^n |u_{i+k}v_i|
    \leq
    \sum_{i=1}^n \frac{C_u}{(i+k)^s}|v_i|
    \leq
    \frac{C_u}{k^s}\sum_{i=1}^n|v_i|,\quad
    \sum_{i=1}^n |u_{i}v_{i+k}| \leq
    \frac{C_v}{k^s}\sum_{i=1}^n|u_i|,
\end{equation*}
\begin{equation*}
    \sum_{i=n+1}^\infty |u_{i+k}v_i|\leq
    \sum_{i=n+1}^\infty\frac{C_u C_v}{i^s(i+k)^s}
    \leq
    \frac{C_u C_v}{k^s}\sum_{i=n+1}^\infty\frac{1}{i^s}
    \leq \frac{C_u C_v}{k^s} \frac{n^{1-s}}{s-1},
    \
    \sum_{i=n+1}^\infty |u_{i}v_{i+k}|\leq \frac{C_u C_v}{k^s} \frac{n^{1-s}}{s-1}.
\end{equation*}
 We also obtain
 \begin{equation*}
     \sum_{i=1}^n |u_{k-i}v_i| \leq
     \frac{C_u}{k^s}\sum_{i=1}^n |v_i|
     \left(\frac{k}{k-i}\right)^s
     =
     \frac{C_u}{k^s}\sum_{i=1}^n |v_i|
     \left(1+\frac{i}{k-i}\right)^s
     \leq
     \frac{C_u}{k^s}\sum_{i=1}^n |v_i|\left(1+\frac{i}{2n+1-i}\right)^s.
 \end{equation*}
 The last inequality follows from fact that the sequence $\{(1+\frac{i}{k-i})\}_{k\geq 2n+1}$ is decreasing with respect to $k.$ Similarly, we have

 \begin{equation*}
     \sum_{i=k-n}^{k-1} |u_{k-i}v_i| =
     \sum_{i=1}^{n} |u_iv_{k-i}|
    \leq
     \frac{C_v}{k^s}\sum_{i=1}^n |u_i|\left(1+\frac{i}{2n+1-i}\right)^s.
 \end{equation*}
 The last infinite sum is estimated as follows
 \begin{align*}
     \sum_{i=n+1}^{k-1 - n}|u_i v_{k-i}|&\leq C_u C_v \sum_{i=n+1}^{k-1 - n} \frac{1}{i^s(k-i)^s} =
     \frac{C_u C_v}{k^s}
     \sum_{i=n+1}^{k-1 - n} \left
     (\frac{1}{i} +\frac{1}{k-i}\right)^s
     \\&\leq
     \frac{2^{s-1} C_u C_v}{k^s}
     \left(
     \sum_{i=n+1}^{k-1 - n}\frac{1}{i^s} +
     \sum_{i=n+1}^{k-1 - n}\frac{1}{(k -i)^s}
     \right)
     \leq
     \frac{2^{s-1} C_u C_v}{k^s} \frac{2n^{-s+1}}{s-1}.
 \end{align*}
Combining the estimates of all six sums yields directly the assertion of the lemma.
\end{proof}

The next lemma is analogous to Lemma \ref{lem:sinTimesSin} and gives us the estimate on the result of multiplication of $u$ and $v$ which are represented in cosine and sine series, respectively.
The first $n+1$ coefficients of $u$ and the first $n$ coefficients of $v$ are given explicitly and the rest of them is expressed by the polynomial decay.

\begin{lemma}\label{lem:cosTimesSin}
Assume that
\begin{equation}
u(x) = u_0 + \sum_{i=1}^\infty u_i\cos(ix), \quad v(x) = \sum_{i=1}^\infty v_i \sin(ix).
\end{equation}
Moreover, assume that for some
$n\in\mathbb{N}$ and $s>1$ the following bounds hold
\begin{equation}
    u_i \in \frac{C_u[-1,1]}{i^s}\quad \textrm{and}\quad
    v_i \in \frac{C_v[-1,1]}{i^s}\quad\text{for $i>n,$}
\end{equation}
where $C_v>0,C_u>0.$
Then
\begin{equation}
(uv)(x) = \sum_{k=1}^\infty (uv)_k\sin(kx),
\end{equation}
with
\begin{equation}
(uv)_k = u_0v_k +
\frac{1}{2}\sum_{i=1}^\infty u_iv_{i+k}
-\frac{1}{2}\sum_{i=1}^\infty u_{i+k}v_i
+\frac{1}{2}\sum_{i=1}^{k-1}  u_{i}v_{k-i},
\end{equation}
and the following estimates on the coefficients of the product hold for $1\leq k\leq 2n$
\begin{equation}
  (uv)_{k}\in u_0v_k +
  \frac{1}{2} \left(
    \sum_{i = 1}^n v_{i+k}u_{i} -
    \sum_{i=  1}^n v_{i}u_{k + i} +
    \sum_{i = 1 }^{k-1}v_{i}u_{k - i}
    \right)
    + C_uC_v\frac{ n^{1-2s}}{ 2s - 1 } [-1,1],
\end{equation}
while  for $k>2n$ we have
\begin{equation}
        (uv)_k \in \frac{D[-1,1]}{k^s},\quad D = |u_0|C_v+
    \frac{1}{2}\left(
    \sum_{i=1}^n (C_u|v_i| + C_v|u_i|)\left(1+\left(\frac{2n+1}{2n+1-i}\right)^s\right)+
    C_uC_v(2+2^s)\frac{n^{1-s}}{s-1}
    \right).
\end{equation}
\end{lemma}
\begin{proof}
The argument is analogous to the proof of Lemma \ref{lem:sinTimesSin}. Namely, we have the following representation of the product
\begin{align*}
    uv &=
    \left(u_0 + \sum_{i_1=1}^\infty u_{i_1}
    \cos(i_1 x)\right)
    \left(\sum_{i_2=1}^\infty v_{i_2} \sin(i_2 x)\right)
    \\
    &=
    \sum_{k=1}^\infty u_0v_k \sin(k x)
     +
    \frac{1}{2}\sum_{i_1,i_2=1}^{\infty} u_{i_1}v_{i_2} \sin(i_2-i_1) +
    \frac{1}{2}\sum_{i_1,i_2=1}^{\infty} u_{i_1}v_{i_2} \sin(i_1+i_2)
    \\&= \sum_{k=1}^\infty
    \left(
    u_0v_k
    + \frac{1}{2}\sum_{i_2-i_1 = k} u_{i_1}v_{i_2}
    -\frac{1}{2}\sum_{i_2-i_1 = -k} u_{i_1}v_{i_2}
    + \frac{1}{2}\sum_{i_2+i_1 = k} u_{i_1}v_{i_2}
    \right)\sin(k x)
\end{align*}
So, for natural $k\geq 1$ we  represent all coefficients using the formulas
\begin{equation}
(uv)_k = u_0v_k +
\frac{1}{2}\sum_{i=1}^\infty u_iv_{i+k}
-\frac{1}{2}\sum_{i=1}^\infty u_{i+k}v_i
+\frac{1}{2}\sum_{i=1}^{k-1}  u_{i}v_{k-i},
\end{equation}
For $1 \leq k\leq 2n$ we write
\begin{equation*}
(uv)_k = u_0v_k +
\frac{1}{2}\sum_{i=1}^n u_iv_{i+k}
-\frac{1}{2}\sum_{i=1}^n u_{i+k}v_i
+\frac{1}{2}\sum_{i=1}^{k-1}  u_{i}v_{k-i}
+\frac{1}{2}\sum_{i=n+1}^\infty (u_iv_{i+k} -u_{i+k}v_i)
.
\end{equation*}
The infinite sums in the above formula are estimated in the same way as in Lemma \ref{lem:sinTimesSin}.
For $k>2n$ we have
%\begin{align*}
%    uv &=
 %   \left(u_0 + \sum_{i_1=1}^\infty u_{i_1}
  %  \cos(i_1 x)\right)
   % \left(\sum_{i_2=1}^\infty v_{i_2} \sin(i_2 x)\right)
    %\\
    %&=
    %\sum_{k=1}^\infty u_0v_k \sin(k x)
    % +
    %\frac{1}{2}\sum_{i_1,i_2=1} u_{i_1}v_{i_2} \sin(i_2-i_1) +
    %\frac{1}{2}\sum_{i_1,i_2=1} u_{i_1}v_{i_2} \sin(i_1+i_2)
    %\\&= \sum_{k=1}^\infty
    %\left(
    %u_0v_k
    %+ \frac{1}{2}\sum_{i_2-i_1 = k} u_{i_1}v_{i_2}
    %-\frac{1}{2}\sum_{i_2-i_1 = -k} u_{i_1}v_{i_2}
    %+ \frac{1}{2}\sum_{i_2+i_1 = k} u_{i_1}v_{i_2}
    %\right)\sin(k x)
%\end{align*}
\begin{align*}
    (uv)_k &= u_0v_k +
    \frac{1}{2}\sum_{i=1}^n u_iv_{i+k}
-\frac{1}{2}\sum_{i=1}^n u_{i+k}v_i
+\frac{1}{2}\sum_{i=n+1}^\infty u_iv_{i+k} -\frac{1}{2}\sum_{i=n+1}^\infty u_{i+k}v_i
\\&+\frac{1}{2}\sum_{i=1}^{n}  u_{i}v_{k-i} +
\frac{1}{2}\sum_{i=k-n}^{k-1}  u_{i}v_{k-i}
+ \frac{1}{2}\sum_{i=n+1}^{k-n-1}u_{i}v_{k-i}.
\end{align*}
It is easy to see that the first component of the above sum can be estimated in the following way $|u_0v_k|\leq \frac{|u_0|C_v}{k^s}.$ The remaining infinite sums are estimated in the same way as in Lemma \ref{lem:sinTimesSin}.
\end{proof}

The following simple lemmas are useful to implement operations on the infinite interval vectors. They can be used when working both with sine and cosine Fourier series.
\begin{lemma}
Assume that the sequence $\{u_i \}_{i=1}^\infty$ satisfies
\begin{equation*}
     u_i\in [u_i^-,u_i^+]\ \
     \text{ for }\ \ i\leq n \ \ \text{ and }\ \
     u_i \in \frac{[C_u^-,C_u^+]}{i^s}\ \
     \text{ for  }\ \ i>n .
\end{equation*}
If $k<n$, then
\begin{equation*}
    u_i \in [D_u^-,D_u^+]\quad\text{for } i>k.
\end{equation*}
where,
\begin{equation*}
    D_u^- =
    \min \{ u_{k+1}^-(k+1)^s,\ldots,u_{n}n^s ,C_u^- \}, \quad
    D_u^+ = \max \{ u_{k+1}^+(k+1)^s,\ldots,u_{n}n^s,C_u^+ \}.
\end{equation*}
\end{lemma}

\begin{lemma}
Assume that sequence $\{u_i \}_{i=1}^\infty$ satisfies
\begin{equation*}
     u_i\in [u_i^-,u_i^+]\ \
     \text{ for }\ \ i\leq n\ \  \text{ and }  \ \
     u_i \in \frac{[C_u^-,C_u^+]}{i^s}\ \
     \text{ for  }\ \ i>k.
\end{equation*}
Then for $s_1<s$ there holds
\begin{equation*}
    u_i \in [D_u^-,D_u^+]\quad\text{for }i>n,
\end{equation*}
where,
\begin{equation*}
    D_u^- =
    \min \left\{ 0,\frac{C_u^-}{(n+1)^{s-s_1}} \right\}, \quad
    D_u^+ =
    \max \left\{0,\frac{C_u^+}{(n+1)^{s-s_1}} \right\}.
\end{equation*}
\end{lemma}
\begin{lemma}\label{lem:add}
Assume that sequences $\{u_i \}_{i=1}^\infty$ and $\{v_i \}_{i=1}^\infty$ satisfy

\begin{equation}
    u_i \in \frac{[C_u^-,C_u^+]}{i^{s_1}},\quad
    v_i \in \frac{[C_v^-,C_v^+]}{i^{s_2}}\quad\text{for $i>n$},
\end{equation}
with the constants $s_1,s_2.$ Then
\begin{itemize}
    \item if $s_1=s_2$ then for $i>n$ we have $u_i+v_i\in
    \frac{[C_v^- + C_v^-,C_u^+ + C_v^+]}{i^s},$
    \item if $s_1 <s_2$ then for $i>n$ we have
    $u_i+v_i\in
    \frac{[C_u^-,C_u^+] + \frac{[0,1]}{(n+1)^{s_2-s_1}}[C_v^-,C_v^+] }{i^s_1},$
    \item if $s_1 >s_2$ then for $i>n$ we have
    $u_i+v_i\in
    \frac{\frac{[0,1]}{(n+1)^{s_1-s_2}}[C_u^-,C_u^+] + [C_v^-,C_v^+] }{i^{s_2}}.$
\end{itemize}
\end{lemma}
\begin{lemma}\label{lem:mult}
Assume that sequences $\{u_i \}_{i=1}^\infty$ and $\{v_i \}_{i=1}^\infty$ satisfy

\begin{equation}
    u_i \in \frac{[C_u^-,C_u^+]}{i^{s_1}}\quad \textrm{and}\quad
    v_i \in \frac{[C_v^-,C_v^+]}{i^{s_2}}\quad\text{for $i>n$},
\end{equation}
with some constants $s_1,s_2.$
Then for $i>n$ we have
$u_iv_i\in \frac{[C_u^-+C_v^-,C_u^+ + C_v^+]}{i^{s_1+s_2}}.$
\end{lemma}

\printbibliography
\end{document}